\documentclass[11pt, reqno]{amsart}
\usepackage{amssymb}
\usepackage{eucal}
\usepackage{amsmath}
\usepackage{amscd}
\usepackage[dvips]{color}
\usepackage{multicol}
\usepackage[all]{xy}           %xypic macro for latex2.09
\usepackage{graphicx}
\usepackage{color}
\usepackage{colordvi}
\usepackage{xspace}
\usepackage{tikz}

\usepackage[active]{srcltx} %SRC Specials for DVI Searching

\begin{document}
\newcommand {\emptycomment}[1]{} %to remove paragraphs

\baselineskip=14pt
\newcommand{\nc}{\newcommand}
\newcommand{\delete}[1]{}
\nc{\mfootnote}[1]{\footnote{#1}} % Use this to show footnotes
\nc{\todo}[1]{\tred{To do:} #1}

%\delete{
\nc{\mlabel}[1]{\label{#1}}  % Use this to suppress names
\nc{\mcite}[1]{\cite{#1}}  % Use this to suppress names
\nc{\mref}[1]{\ref{#1}}  % Use this to suppress names
\nc{\mbibitem}[1]{\bibitem{#1}} % Use this to show number
%}

\delete{
\nc{\mlabel}[1]{\label{#1}  % Use the next two lines to show names
{\hfill \hspace{1cm}{\bf{{\ }\hfill(#1)}}}}
\nc{\mcite}[1]{\cite{#1}{{\bf{{\ }(#1)}}}}  % Use this lines to show names
\nc{\mref}[1]{\ref{#1}{{\bf{{\ }(#1)}}}}  % Use this lines to show names
\nc{\mbibitem}[1]{\bibitem[\bf #1]{#1}} % Use this to show name
}

%%%%%%%%%%%%%%%%%%%%%%%% Statements
\newtheorem{thm}{Theorem}[section]
\newtheorem{lem}[thm]{Lemma}
\newtheorem{cor}[thm]{Corollary}
\newtheorem{pro}[thm]{Proposition}
\newtheorem{ex}[thm]{Example}
\newtheorem{rmk}[thm]{Remark}
\newtheorem{defi}[thm]{Definition}
\newtheorem{pdef}[thm]{Proposition-Definition}
\newtheorem{condition}[thm]{Condition}

\renewcommand{\labelenumi}{{\rm(\alph{enumi})}}
\renewcommand{\theenumi}{\alph{enumi}}

\nc{\tred}[1]{\textcolor{red}{#1}}
\nc{\tblue}[1]{\textcolor{blue}{#1}}
\nc{\tgreen}[1]{\textcolor{green}{#1}}
\nc{\tpurple}[1]{\textcolor{purple}{#1}}
\nc{\btred}[1]{\textcolor{red}{\bf #1}}
\nc{\btblue}[1]{\textcolor{blue}{\bf #1}}
\nc{\btgreen}[1]{\textcolor{green}{\bf #1}}
\nc{\btpurple}[1]{\textcolor{purple}{\bf #1}}

\nc{\YLM}[1]{\textcolor{red}{Lamei:#1}}
\nc{\Hyy}[1]{\textcolor{blue}{Yanyong: #1}}
%\nc{\lit}[2]{\textcolor{blue}{#1}{ \textcolor{purple}{(#2)}}}

%%%%%%%%%%%%%% Matrix symbols.

\nc{\twovec}[2]{\left(\begin{array}{c} #1 \\ #2\end{array} \right )}
\nc{\threevec}[3]{\left(\begin{array}{c} #1 \\ #2 \\ #3 \end{array}\right )}
\nc{\twomatrix}[4]{\left(\begin{array}{cc} #1 & #2\\ #3 & #4 \end{array} \right)}
\nc{\threematrix}[9]{{\left(\begin{matrix} #1 & #2 & #3\\ #4 & #5 & #6 \\ #7 & #8 & #9 \end{matrix} \right)}}
\nc{\twodet}[4]{\left|\begin{array}{cc} #1 & #2\\ #3 & #4 \end{array} \right|}

\nc{\rk}{\mathrm{r}}
\newcommand{\g}{\mathfrak g}
\newcommand{\h}{\mathfrak h}
\newcommand{\pf}{\noindent{$Proof$.}\ }
\newcommand{\frkg}{\mathfrak g}
\newcommand{\frkh}{\mathfrak h}
\newcommand{\Id}{\rm{Id}}
\newcommand{\gl}{\mathfrak {gl}}
\newcommand{\ad}{\mathrm{ad}}
\newcommand{\add}{\frka\frkd}
\newcommand{\frka}{\mathfrak a}
\newcommand{\frkb}{\mathfrak b}
\newcommand{\frkc}{\mathfrak c}
\newcommand{\frkd}{\mathfrak d}
\newcommand {\comment}[1]{{\marginpar{*}\scriptsize\textbf{Comments:} #1}}
%%%%%%%%%%%%%%%%%%%%%%% symbols

\nc{\gensp}{V} % space of generators
\nc{\relsp}{\Lambda} %space of relations
\nc{\leafsp}{X}    %decoration space for leaves
\nc{\treesp}{\overline{\calt}} % space of labeled trees

\nc{\vin}{{\mathrm Vin}}    %decoration set of indices
\nc{\lin}{{\mathrm Lin}}    %decoration set of leaves

\nc{\gop}{{\,\omega\,}}     % generic binary operation
\nc{\gopb}{{\,\nu\,}}
\nc{\svec}[2]{{\tiny\left(\begin{matrix}#1\\
#2\end{matrix}\right)\,}}  % column vector
\nc{\ssvec}[2]{{\tiny\left(\begin{matrix}#1\\
#2\end{matrix}\right)\,}} % subscript column vector

\nc{\typeI}{local cocycle $3$-Lie bialgebra\xspace}
\nc{\typeIs}{local cocycle $3$-Lie bialgebras\xspace}
\nc{\typeII}{double construction $3$-Lie bialgebra\xspace}
\nc{\typeIIs}{double construction $3$-Lie bialgebras\xspace}

\nc{\bia}{{$\mathcal{P}$-bimodule ${\bf k}$-algebra}\xspace}
\nc{\bias}{{$\mathcal{P}$-bimodule ${\bf k}$-algebras}\xspace}

\nc{\rmi}{{\mathrm{I}}}
\nc{\rmii}{{\mathrm{II}}}
\nc{\rmiii}{{\mathrm{III}}}
\nc{\pr}{{\mathrm{pr}}}
\newcommand{\huaA}{\mathcal{A}}

\nc{\pll}{\beta}
\nc{\plc}{\epsilon}

\nc{\ass}{{\mathit{Ass}}}
\nc{\lie}{{\mathit{Lie}}}
\nc{\comm}{{\mathit{Comm}}}
\nc{\dend}{{\mathit{Dend}}}
\nc{\zinb}{{\mathit{Zinb}}}
\nc{\tdend}{{\mathit{TDend}}}
\nc{\prelie}{{\mathit{preLie}}}
\nc{\postlie}{{\mathit{PostLie}}}
\nc{\quado}{{\mathit{Quad}}}
\nc{\octo}{{\mathit{Octo}}}
\nc{\ldend}{{\mathit{ldend}}}
\nc{\lquad}{{\mathit{LQuad}}}

 \nc{\adec}{\check{;}} \nc{\aop}{\alpha}
\nc{\dftimes}{\widetilde{\otimes}} \nc{\dfl}{\succ} \nc{\dfr}{\prec}
\nc{\dfc}{\circ} \nc{\dfb}{\bullet} \nc{\dft}{\star}
\nc{\dfcf}{{\mathbf k}} \nc{\apr}{\ast} \nc{\spr}{\cdot}
\nc{\twopr}{\circ} \nc{\tspr}{\star} \nc{\sempr}{\ast}
\nc{\disp}[1]{\displaystyle{#1}}
\nc{\bin}[2]{ (_{\stackrel{\scs{#1}}{\scs{#2}}})}  %binomial coeff
\nc{\binc}[2]{ \left (\!\! \begin{array}{c} \scs{#1}\\
    \scs{#2} \end{array}\!\! \right )}  %binomial coeff
\nc{\bincc}[2]{  \left ( {\scs{#1} \atop
    \vspace{-.5cm}\scs{#2}} \right )}  %binomial coeff
\nc{\sarray}[2]{\begin{array}{c}#1 \vspace{.1cm}\\ \hline
    \vspace{-.35cm} \\ #2 \end{array}}
\nc{\bs}{\bar{S}} \nc{\dcup}{\stackrel{\bullet}{\cup}}
\nc{\dbigcup}{\stackrel{\bullet}{\bigcup}} \nc{\etree}{\big |}
\nc{\la}{\longrightarrow} \nc{\fe}{\'{e}} \nc{\rar}{\rightarrow}
\nc{\dar}{\downarrow} \nc{\dap}[1]{\downarrow
\rlap{$\scriptstyle{#1}$}} \nc{\uap}[1]{\uparrow
\rlap{$\scriptstyle{#1}$}} \nc{\defeq}{\stackrel{\rm def}{=}}
\nc{\dis}[1]{\displaystyle{#1}} \nc{\dotcup}{\,
\displaystyle{\bigcup^\bullet}\ } \nc{\sdotcup}{\tiny{
\displaystyle{\bigcup^\bullet}\ }} \nc{\hcm}{\ \hat{,}\ }
\nc{\hcirc}{\hat{\circ}} \nc{\hts}{\hat{\shpr}}
\nc{\lts}{\stackrel{\leftarrow}{\shpr}}
\nc{\rts}{\stackrel{\rightarrow}{\shpr}} \nc{\lleft}{[}
\nc{\lright}{]} \nc{\uni}[1]{\tilde{#1}} \nc{\wor}[1]{\check{#1}}
\nc{\free}[1]{\bar{#1}} \nc{\den}[1]{\check{#1}} \nc{\lrpa}{\wr}
\nc{\curlyl}{\left \{ \begin{array}{c} {} \\ {} \end{array}
    \right .  \!\!\!\!\!\!\!}
\nc{\curlyr}{ \!\!\!\!\!\!\!
    \left . \begin{array}{c} {} \\ {} \end{array}
    \right \} }
\nc{\leaf}{\ell}       % number of leafs
\nc{\longmid}{\left | \begin{array}{c} {} \\ {} \end{array}
    \right . \!\!\!\!\!\!\!}
\nc{\ot}{\otimes} \nc{\sot}{{\scriptstyle{\ot}}}
\nc{\otm}{\overline{\ot}}
\nc{\ora}[1]{\stackrel{#1}{\rar}}
\nc{\ola}[1]{\stackrel{#1}{\la}}%${\Bbb Z}$
\nc{\pltree}{\calt^\pl}
\nc{\epltree}{\calt^{\pl,\NC}}
\nc{\rbpltree}{\calt^r}
\nc{\scs}[1]{\scriptstyle{#1}} \nc{\mrm}[1]{{\rm #1}}
\nc{\dirlim}{\displaystyle{\lim_{\longrightarrow}}\,}
\nc{\invlim}{\displaystyle{\lim_{\longleftarrow}}\,}
\nc{\mvp}{\vspace{0.5cm}} \nc{\svp}{\vspace{2cm}}
\nc{\vp}{\vspace{8cm}} \nc{\proofbegin}{\noindent{\bf Proof: }}
%\nc{\proofbegin}{\begin{proof}} % AMS command
\nc{\proofend}{$\blacksquare$ \vspace{0.5cm}}
%\nc{\proofend}{\end{proof}} %AMS command
\nc{\freerbpl}{{F^{\mathrm RBPL}}}
\nc{\sha}{{\mbox{\cyr X}}}  %used to be \cyr
\nc{\ncsha}{{\mbox{\cyr X}^{\mathrm NC}}} \nc{\ncshao}{{\mbox{\cyr
X}^{\mathrm NC,\,0}}}
\nc{\shpr}{\diamond}    %Shuffle product
\nc{\shprm}{\overline{\diamond}}    %Shuffle product
\nc{\shpro}{\diamond^0}    %Shuffle product
\nc{\shprr}{\diamond^r}     %product on controlled trees
\nc{\shpra}{\overline{\diamond}^r}
\nc{\shpru}{\check{\diamond}} \nc{\catpr}{\diamond_l}
\nc{\rcatpr}{\diamond_r} \nc{\lapr}{\diamond_a}
\nc{\sqcupm}{\ot}
\nc{\lepr}{\diamond_e} \nc{\vep}{\varepsilon} \nc{\labs}{\mid\!}
\nc{\rabs}{\!\mid} \nc{\hsha}{\widehat{\sha}}
\nc{\lsha}{\stackrel{\leftarrow}{\sha}}
\nc{\rsha}{\stackrel{\rightarrow}{\sha}} \nc{\lc}{\lfloor}
\nc{\rc}{\rfloor}
\nc{\tpr}{\sqcup}
\nc{\nctpr}{\vee}
\nc{\plpr}{\star}
\nc{\rbplpr}{\bar{\plpr}}
\nc{\sqmon}[1]{\langle #1\rangle}
\nc{\forest}{\calf}
\nc{\altx}{\Lambda_X} \nc{\vecT}{\vec{T}} \nc{\onetree}{\bullet}
\nc{\Ao}{\check{A}}
\nc{\seta}{\underline{\Ao}}
\nc{\deltaa}{\overline{\delta}}
\nc{\trho}{\tilde{\rho}}

\nc{\rpr}{\circ}
%\nc{\apr}{\cdot}
\nc{\dpr}{{\tiny\diamond}}
\nc{\rprpm}{{\rpr}}

%%%%%%%%%%%%%%%%%%%%% roman fonts, in alphabetic order
\nc{\mmbox}[1]{\mbox{\ #1\ }} \nc{\ann}{\mrm{ann}}
\nc{\Aut}{\mrm{Aut}} \nc{\can}{\mrm{can}}
\nc{\twoalg}{{two-sided algebra}\xspace}
\nc{\colim}{\mrm{colim}}
\nc{\Cont}{\mrm{Cont}} \nc{\rchar}{\mrm{char}}
\nc{\cok}{\mrm{coker}} \nc{\dtf}{{R-{\rm tf}}} \nc{\dtor}{{R-{\rm
tor}}}
\renewcommand{\det}{\mrm{det}}
\nc{\depth}{{\mrm d}}
\nc{\Div}{{\mrm Div}} \nc{\End}{\mrm{End}} \nc{\Ext}{\mrm{Ext}}
\nc{\Fil}{\mrm{Fil}} \nc{\Frob}{\mrm{Frob}} \nc{\Gal}{\mrm{Gal}}
\nc{\GL}{\mrm{GL}} \nc{\Hom}{\mrm{Hom}} \nc{\hsr}{\mrm{H}}
\nc{\hpol}{\mrm{HP}} \nc{\id}{\mrm{id}} \nc{\im}{\mrm{im}}
\nc{\incl}{\mrm{incl}} \nc{\length}{\mrm{length}}
\nc{\LR}{\mrm{LR}} \nc{\mchar}{\rm char} \nc{\NC}{\mrm{NC}}
\nc{\mpart}{\mrm{part}} \nc{\pl}{\mrm{PL}}
\nc{\ql}{{\QQ_\ell}} \nc{\qp}{{\QQ_p}}
\nc{\rank}{\mrm{rank}} \nc{\rba}{\rm{RBA }} \nc{\rbas}{\rm{RBAs }}
\nc{\rbpl}{\mrm{RBPL}}
\nc{\rbw}{\rm{RBW }} \nc{\rbws}{\rm{RBWs }} \nc{\rcot}{\mrm{cot}}
\nc{\rest}{\rm{controlled}\xspace}
\nc{\rdef}{\mrm{def}} \nc{\rdiv}{{\rm div}} \nc{\rtf}{{\rm tf}}
\nc{\rtor}{{\rm tor}} \nc{\res}{\mrm{res}} \nc{\SL}{\mrm{SL}}
\nc{\Spec}{\mrm{Spec}} \nc{\tor}{\mrm{tor}} \nc{\Tr}{\mrm{Tr}}
\nc{\mtr}{\mrm{sk}}

%%%%%%%%%%%%%%%%%% bold face
\nc{\ab}{\mathbf{Ab}} \nc{\Alg}{\mathbf{Alg}}
\nc{\Algo}{\mathbf{Alg}^0} \nc{\Bax}{\mathbf{Bax}}
\nc{\Baxo}{\mathbf{Bax}^0} \nc{\RB}{\mathbf{RB}}
\nc{\RBo}{\mathbf{RB}^0} \nc{\BRB}{\mathbf{RB}}
\nc{\Dend}{\mathbf{DD}} \nc{\bfk}{{\bf k}} \nc{\bfone}{{\bf 1}}
\nc{\base}[1]{{a_{#1}}} \nc{\detail}{\marginpar{\bf More detail}
    \noindent{\bf Need more detail!}
    \svp}
\nc{\Diff}{\mathbf{Diff}} \nc{\gap}{\marginpar{\bf
Incomplete}\noindent{\bf Incomplete!!}
    \svp}
\nc{\FMod}{\mathbf{FMod}} \nc{\mset}{\mathbf{MSet}}
\nc{\rb}{\mathrm{RB}} \nc{\Int}{\mathbf{Int}}
\nc{\Mon}{\mathbf{Mon}}
%\nc{\remark}{\noindent{\bf Remark: }}
\nc{\remarks}{\noindent{\bf Remarks: }}
\nc{\OS}{\mathbf{OS}} %free operated semigroup
\nc{\Rep}{\mathbf{Rep}}
\nc{\Rings}{\mathbf{Rings}} \nc{\Sets}{\mathbf{Sets}}
\nc{\DT}{\mathbf{DT}}

%%%%%%%%%%%%%%%%%%%Bbb fonts
\nc{\BA}{{\mathbb A}} \nc{\CC}{{\mathbb C}} \nc{\DD}{{\mathbb D}}
\nc{\EE}{{\mathbb E}} \nc{\FF}{{\mathbb F}} \nc{\GG}{{\mathbb G}}
\nc{\HH}{{\mathbb H}} \nc{\LL}{{\mathbb L}} \nc{\NN}{{\mathbb N}}
\nc{\QQ}{{\mathbb Q}} \nc{\RR}{{\mathbb R}} \nc{\BS}{{\mathbb{S}}} \nc{\TT}{{\mathbb T}}
\nc{\VV}{{\mathbb V}} \nc{\ZZ}{{\mathbb Z}}

%%%%%%%%%%%%%%%%%%% cal fonts

\nc{\calao}{{\mathcal A}} \nc{\cala}{{\mathcal A}}
\nc{\calc}{{\mathcal C}} \nc{\cald}{{\mathcal D}}
\nc{\cale}{{\mathcal E}} \nc{\calf}{{\mathcal F}}
\nc{\calfr}{{{\mathcal F}^{\,r}}} \nc{\calfo}{{\mathcal F}^0}
\nc{\calfro}{{\mathcal F}^{\,r,0}} \nc{\oF}{\overline{F}}
\nc{\calg}{{\mathcal G}} \nc{\calh}{{\mathcal H}}
\nc{\cali}{{\mathcal I}} \nc{\calj}{{\mathcal J}}
\nc{\call}{{\mathcal L}} \nc{\calm}{{\mathcal M}}
\nc{\caln}{{\mathcal N}} \nc{\calo}{{\mathcal O}}
\nc{\calp}{{\mathcal P}} \nc{\calq}{{\mathcal Q}} \nc{\calr}{{\mathcal R}}
\nc{\calt}{{\mathcal T}} \nc{\caltr}{{\mathcal T}^{\,r}}
\nc{\calu}{{\mathcal U}} \nc{\calv}{{\mathcal V}}
\nc{\calw}{{\mathcal W}} \nc{\calx}{{\mathcal X}}
\nc{\CA}{\mathcal{A}}

%%%%%%%%%%%%%%%%%%  frak fonts
\nc{\fraka}{{\mathfrak a}} \nc{\frakB}{{\mathfrak B}}
\nc{\frakb}{{\mathfrak b}} \nc{\frakd}{{\mathfrak d}}
\nc{\oD}{\overline{D}}
\nc{\frakF}{{\mathfrak F}} \nc{\frakg}{{\mathfrak g}}
\nc{\frakm}{{\mathfrak m}} \nc{\frakM}{{\mathfrak M}}
\nc{\frakMo}{{\mathfrak M}^0} \nc{\frakp}{{\mathfrak p}}
\nc{\frakS}{{\mathfrak S}} \nc{\frakSo}{{\mathfrak S}^0}
\nc{\fraks}{{\mathfrak s}} \nc{\os}{\overline{\fraks}}
\nc{\frakT}{{\mathfrak T}}
\nc{\oT}{\overline{T}}
%\nc{\frakx}{{\mathfrak x}}
\nc{\frakX}{{\mathfrak X}} \nc{\frakXo}{{\mathfrak X}^0}
\nc{\frakx}{{\mathbf x}}
%\nc{\frakTxo}{{\frakTx}^0}
\nc{\frakTx}{\frakT}      %All rooted trees, correspond to \ncsha(X)
\nc{\frakTa}{\frakT^a}        % rooted trees for \ncsha(A)
\nc{\frakTxo}{\frakTx^0}   % rooted trees for \ncshao(X)
\nc{\caltao}{\calt^{a,0}}   % rooted trees for \ncshao(A)
\nc{\ox}{\overline{\frakx}} \nc{\fraky}{{\mathfrak y}}
\nc{\frakz}{{\mathfrak z}} \nc{\oX}{\overline{X}}

\font\cyr=wncyr10

\nc{\redtext}[1]{\textcolor{red}{#1}}
%\nc{\li}[1]{\textcolor{red}{#1}}

%%%%%%%%%%%%%%%%%%%%%%%%%%%%%%%%%%%%%%%%%%%%%%%%%%%%%%%%%%%%%%%%%%

%\begin{document}
\title{Unified products of Leibniz conformal algebras}

\author{Yanyong Hong}
\address{Department of Mathematics, Hangzhou Normal University,
Hangzhou, 311121, China}
\email{hongyanyong2008@yahoo.com}

\author{Lamei Yuan}
\address{Department of Mathematics, Harbin Institute of Technology, Harbin, 150001, P.R.China}
\email{lmyuan@hit.edu.cn}

\subjclass[2010]{}
\keywords{Leibniz conformal algebra, unified product, crossed product, bicrossed product}

\begin{abstract}
The aim of this paper is to provide an answer to the $\mathbb{C}[\partial]$-split extending structures problem for Leibniz conformal algebras, which asks that how to describe all Leibniz conformal algebra structures on $E=R\oplus Q$ up to an isomorphism such that $R$ is a Leibniz conformal subalgebra.
%and $Q$ is a $\mathbb{C}[\partial]$-module.
For this purpose, an unified product of Leibniz conformal algebras is introduced. Using this tool, two cohomological type objects are constructed to classify all such extending structures up to an isomorphism. Then this general theory is applied to the special case when $R$ is a free $\mathbb{C}[\partial]$-module and $Q$ is a free $\mathbb{C}[\partial]$-module of rank one. Finally, the twisted product, crossed product and bicrossed product between two Leibniz conformal algebras are introduced as special cases of the unified product, and some examples are given.
%In this paper, the definition of unified product of Leibniz conformal algebras is introduced, which is a very general product including twisted product, crossed product, bicrossed product and so on. Applying the unified product, we give a cohomological type object to characterize the $\mathbb{C}[\partial]$-split extending structures problem, which asks that how to describe all Leibniz conformal algebra structures on $E=R\oplus Q$ such that $R$ is a Leibniz conformal subalgebra up to isomorphsim, where $Q$ is a $\mathbb{C}[\partial]$-module. Using this general theory, some special cases are investigated, for example, when $R$ is a free $\mathbb{C}[\partial]$-module and $Q$ is a free $\mathbb{C}[\partial]$-module of rank one. Moreover, we study crossed product and bicrossed product in detail and some examples are given.
\end{abstract}

\maketitle
{\footnote{The second author is the corresponding author.}}
\section{Introduction}
As non-commutative analogues of Lie algebras, Leibniz algebras were first introduced by Bloh in \cite{B} and reintroduced later by Loday in \cite{L} during their study on periodicity phenomena in algebraic $K$-theory. The name of left (resp. right) Leibniz algebra comes from that the left (resp. right) multiplication is a derivation. Conformal algebras were introduced in \cite{K1} in order to give an
axiomatic description of the operator product expansion (or rather
its Fourier transform) of chiral fields in conformal field theory.
Leibniz conformal algebras were introduced in \cite{BKV} (see also \cite{DK1}) and their cohomology theory was investigated in \cite{BKV,Z}. It was shown in \cite{BK} that Leibniz conformal algebras are closely related to field algebras, which are non-commutative generalizations of vertex algebras. The notion of a Leibniz pseudoalgebra was introduced and studied in \cite{Wu}. As a special class of Leibniz conformal algebras, quadratic Leibniz conformal algebras were studied
%and their central extensions were also investigated
in \cite{ZH}. Recently, all torsion-free Leibniz conformal algebras of rank $2$ were classified in \cite{Wu1}.

In this paper, we will consider a general extension theory of Leibniz conformal algebras. Let $R$ be a Leibniz conformal algebra and $Q$ a $\mathbb{C}[\partial]$-module. It was stated in \cite{BKV} that when $R$ is  abelian and $Q$ is a Leibniz conformal algebra, all $\mathbb{C}[\partial]$-split central extensions of $Q$ by $R$ (up to equivalence) can be characterized by the second cohomology group $H^2(Q,R)$. We will generalize this problem to a more general case:

{\bf Set $E=R\oplus Q$, where the direct sum is the sum of $\mathbb{C}[\partial]$-modules. Describe and classify all Leibniz conformal algebra structures on $E$ such that $R$ is a subalgebra of $E$ up to isomorphisms whose restrictions on $R$ are the identity maps.}

The above problem is called the {\bf $\mathbb{C}[\partial]$-split extending structures problem}. It is a very general question including the following interesting algebra problem, which appears in the homological algebra theory:

\emph{For two Leibniz conformal algebras $R$ and $Q$, describe and classify (up to equivalence)
all $\mathbb{C}[\partial]$-split exact sequences of Leibniz conformal algebras as follows:
\begin{eqnarray}\label{en1}
\xymatrix@C=0.5cm{
  0 \ar[r] & R \ar[rr]^{i} && E\ar[rr]^{\pi} && Q\ar[r] & 0 }.
\end{eqnarray}}
This problem is called the {\bf $\mathbb{C}[\partial]$-split extension problem}. It is a direct generalization of the central extension problem.

The $\mathbb{C}[\partial]$-split extending structures problem for Leibniz algebras, Lie conformal algebras and associative conformal algebras were solved in \cite{AM, HS, H}, respectively. The purpose of the present paper is to do the same for Leibniz conformal algebras. In this paper, we will provide an answer to the $\mathbb{C}[\partial]$-split extending structures problem as follows: first we will describe all Leibniz conformal algebra structures on $E$ which contains $R$ as a subalgebra by defining a unified product of  $R$ and a $\mathbb{C}[\partial]$-module $Q$; then we will classify them up to a Leibniz conformal algebra isomorphism $\varphi: E\rightarrow E$ such that $\varphi$ acts as the identity on $R$. In the case when $R=0$, the $\mathbb{C}[\partial]$-split extending structures problem is equivalent to classifying all Leibniz conformal algebras of arbitrary rank. It is very difficult, even when the rank is 3. For this reason, we will always assume that $R\neq 0$.

%(In this paper, we introduce the definition of unified product of a Leibniz conformal algebra $R$ and a $\mathbb{C}[\partial]$-module $Q$ and construct a cohomological type object to describe all Leibniz conformal algebra structures on $E$ up to isomorphism by using the tool of unified product. With this general theory, we also study the case when  $R$ is a free $\mathbb{C}[\partial]$-module and $Q$ is a free $\mathbb{C}[\partial]$-module of rank one. Moreover, as special cases of unified product, we also introduce twisted product, crossed product and bicrossed product. It is shown that crossed product can be used to give a theoretical answer to the $\mathbb{C}[\partial]$-split extensions problem. Some interesting results in these special cases are also obtained. Note that when $R=0$, the $\mathbb{C}[\partial]$-split extending structures problem is equivalent to classifying all Leibniz conformal algebras of arbitrary ranks. Obviously, it is very hard, even when the rank is 3. Therefore, we always assume $R\neq 0$.

 %Such a problem is a basic algebra problem for Leibniz conformal algebras and it is important for studying Leibniz conformal algebras. Therefore, it is meaningful and the obtained results will be useful for investigating the structure theory of Leibniz conformal algebras.

The rest of the paper is organized as follows. In Section 2, the definitions of Lie conformal algebras and Leibniz conformal algebras are recalled. Also, some necessary concepts including conformal linear map, conformal bilinear map, module, conformal derivation and conformal anti-derivation of a Leibniz conformal algebra are reviewed. %which are all related with the $\mathbb{C}[\partial]$-split extending structures problem.
In Section 3, we introduce the notion of a unified product of Leibniz conformal algebras. By using this tool, we  construct a cohomological type object to give a theoretical answer to the $\mathbb{C}[\partial]$-split extending structures problem. In Section 4, we apply the general theory developed in Section 3 to the special case when $R$ is a free $\mathbb{C}[\partial]$-module and $Q$ is a free $\mathbb{C}[\partial]$-module of rank 1.
In Section 5, some particular cases of the unified product including twisted product, crossed product and bicrossed product are discussed. Also, some examples are given in details.

Throughout this paper, in addition to the standard notations $\mathbb{C}$, $\mathbb{N}$ and $\mathbb{Z}$, we also assume that all vector spaces, algebras and tensors are over $\mathbb{C}$. % are denoted by $\otimes$.
Moreover, the space of polynomials in $\lambda$ with coefficients in a vector space $A$ is denoted by $A[\lambda]$.

\section{Preliminaries}
In this section, some concepts related to Leibniz conformal algebras are recalled.
\begin{defi}
A \emph{conformal algebra} $R$ is a $\mathbb{C}[\partial]$-module  endowed with a $\mathbb{C}$-bilinear map
\begin{equation*}\label{251}
R\times R\rightarrow R[\lambda], ~~~~~~a\times b\mapsto [a_\lambda b],
\end{equation*}
satisfying the following axiom~($a, b\in R$):
\begin{eqnarray*}\label{251}
Conformal~~sesquilinearity:~~~~[(\partial a)_\lambda b]=-\lambda[a_\lambda b],~~~
[a_\lambda \partial b]=(\partial+\lambda)[a_\lambda b].
\end{eqnarray*}

A \emph{Lie conformal algebra} $(R,[\cdot_{\lambda}\cdot])$ is a conformal algebra satisfying ~($a, b, c\in R$):
\begin{eqnarray}
&&Skew~~symmetry:~~~~[a_\lambda b]=-[b_{-\lambda-\partial} a]\label{axiom1};\\
&&Jacobi~~identity:~~~~[a_\lambda [b_\mu c]]=[[a_\lambda b]_{\lambda+\mu} c]+[b_\mu[a_\lambda c]]\label{axiom2}.
\end{eqnarray}

A \emph{Leibniz conformal algebra} $(R,[\cdot_\lambda \cdot])$ is a conformal algebra satisfying the Jacobi identity \eqref{axiom2}.
%\begin{eqnarray*}
%Leibniz~~identity:~~~~[a_\lambda [b _\mu c]]=[[a_\lambda b]_{\lambda+\mu} c]+[b_\mu [a_\lambda
%c]],
%\end{eqnarray*}
%for any $a$, $b$, $c\in R$.
\end{defi}

A Lie (or Leibniz) conformal algebra $(R,[\cdot_\lambda \cdot])$ is called \emph{finite}, if it is a finitely generated
$\mathbb{C}[\partial]$-module; otherwise, it is said to be \emph{infinite}.

\begin{rmk}
Obviously, all Lie conformal algebras are Leibniz conformal algebras.
\end{rmk}

The following is straightforward.
\begin{pro}
Let $R=\mathbb{C}[\partial]x$ be a Leibniz conformal algebra which is a free $\mathbb{C}[\partial]$-module of rank 1. Then $R$ is either abelian or isomorphic to the Virasoro Lie conformal algebra, namely, $[x_\lambda x]=(\partial+2\lambda)x$.
\end{pro}
%\begin{proof}
%It is easy to get.
%\end{proof}

Suppose that $(R,[\cdot_\lambda \cdot])$ is a Leibniz conformal algebra. For $a\in R$, if $[a_\lambda b]=0$ (resp. $[b_\lambda a]=0$) for any $b\in R$, then
$a$ is called a \emph{left-central element} (resp. \emph{right-central element}) of $R$. The set of all left-central elements (resp. right-central elements) of $R$ is called \emph{ the left-center} (resp. \emph{right-center}) of $R$. %Similarly, for some fixed element $a\in R$, if for any $b\in R$, $[b_\lambda a]=0$, then
%$a$ is called \emph{ a right-central element} of $R$. The set of all right-central elements of $R$ is called \emph{ the right-center} of $R$.

\begin{defi}
Let $U$ and $V$ be two $\mathbb{C}[\partial]$-modules. A \emph{left conformal linear map} from $U$ to $V$ is a $\mathbb{C}$-linear map %$a: U\rightarrow V[\lambda]$, denoted by
$a_\lambda: U\rightarrow V[\lambda]$, such that $a_\lambda(\partial u)=-\lambda a_\lambda u$, $\forall \ u\in U$. Similarly, a \emph{right conformal linear map} from $U$ to $V$ is a $\mathbb{C}$-linear map %$a: U\rightarrow V[\lambda]$, denoted by
$a_\lambda: U\rightarrow V[\lambda]$, such that $a_\lambda(\partial u)=(\partial+\lambda)a_\lambda u$, $\forall \ u\in U$. A right conformal linear map is usually called a conformal linear map in short.

Moreover, let $W$ also be a $\mathbb{C}[\partial]$-module. A \emph{conformal bilinear map} from $U\times V$ to $W$ is a $\mathbb{C}$-bilinear map %$f: U\times V \rightarrow W[\lambda]$, denoted by
$f_\lambda: U\times V\rightarrow W[\lambda]$, such that $f_\lambda(\partial u,v)=-\lambda f_\lambda(u,v)$ and $f_\lambda( u,\partial v)= (\partial+\lambda)f_\lambda(u,v)$, $\forall \ u\in U, \ v\in V$.
\end{defi}

\begin{defi}
A \emph{module} $Q$ over a Leibniz conformal algebra $(R, [\cdot_\lambda\cdot])$ is a $\mathbb{C}[\partial]$-module endowed with two conformal bilinear maps
$R\times Q\longrightarrow Q[\lambda]$, $(a, x)\mapsto a\rightharpoonup_\lambda x$ and $Q\times R\longrightarrow Q[\lambda]$, $(x, a)\mapsto x\lhd_\lambda a$ satisfying the following axioms $(a, b\in R, x\in Q)$:
\begin{eqnarray}
a\rightharpoonup_\lambda (b\rightharpoonup_\mu x)=[a_\lambda b]\rightharpoonup_{\lambda+\mu}x
+b\rightharpoonup_\mu(a\rightharpoonup_\lambda x),\\
a\rightharpoonup_\lambda(x\lhd_\mu b)=(a\rightharpoonup_\lambda x)\lhd_{\lambda+\mu}b+x\lhd_\mu[a_\lambda b],\\
x\lhd_\lambda[a_\mu b]=(x\lhd_\lambda a)\lhd_{\lambda+\mu}b+a\rightharpoonup_\mu(x\lhd_\lambda b).
\end{eqnarray}
We also denote it by $(Q,\rightharpoonup_\lambda, \lhd_\lambda)$.
\end{defi}
\begin{defi}
Let $(R, [\cdot_\lambda\cdot])$ be a Leibniz conformal algebra. A conformal linear map $D_\lambda:R\longrightarrow R[\lambda]$ is called a \emph{conformal derivation} if
\begin{equation*}
D_\lambda([a_\mu b])=[(D_\lambda a)_{\lambda+\mu} b]+[a_\mu (D_\lambda b)],\ \ \mbox{ \ for all $a,b\in R$.}
\end{equation*}

A left conformal linear map $D_\lambda:R\longrightarrow R[\lambda]$ is called a \emph{conformal anti-derivation} if
\begin{equation*}
D_{\lambda+\mu}([a_\lambda b])=[a_\lambda (D_\mu b)]-[b_\mu (D_\lambda a)],\ \ \mbox{ \ for all $a,b\in R$.}
\end{equation*}
\end{defi}
It is easy to see that for any $a\in R$, the map $({\rm ad}_a)_\lambda$, defined by $({\rm ad}_a)_\lambda b= [a\, {}_\lambda\, b]$ for any $b\in R$, is a conformal derivation of $R$. All conformal derivations of this kind are called \emph{inner conformal derivations}. Denote by $CDer(R)$
and $CInn(R)$ the vector spaces of all conformal derivations and inner conformal derivations of $R$, respectively.

Similarly, for any $a\in R$, the map $({\rm Ad}_a)_\lambda$, defined by $({\rm Ad}_a)_\lambda b= [b\, {}_\lambda\, a]$ for any $b\in R$, is a conformal anti-derivation of $R$. All conformal anti-derivations of this kind are called \emph{inner conformal anti-derivations}.

Finally, we introduce the following definition which is important for studying
the $\mathbb{C}[\partial]$-split extending structures problem.
\begin{defi}
Let $(R, [\cdot_\lambda\cdot])$ be a Leibniz conformal algebra, $Q$ a $\mathbb{C}[\partial]$-module and $E=R\oplus Q$ where the direct sum is the sum of $\mathbb{C}[\partial]$-modules. Set $\varphi: E\rightarrow E$ be a $\mathbb{C}[\partial]$-module homomorphism. We consider  the following
diagram:
$$\xymatrix{ {R}\ar[d]^{Id}\ar[r]^{i}& {E}\ar[d]^{\varphi}\ar[r]^{\pi} & {Q}\ar[d]^{Id} \\
{R}\ar[r]^{i}&{E}\ar[r]^{\pi} &{Q} }$$
where $\pi: E\rightarrow Q$ is the natural projection of $E=R\oplus Q$ onto $Q$ and
$i: R\rightarrow E$ is the inclusion map. If the left square
(resp. the right square) of the  above diagram is commutative,
we call that
$\varphi: E\rightarrow E$ \emph{stabilizes} $R$ (resp. \emph{co-stabilizes} $Q$).

Let $[\cdot_\lambda\cdot]$ and $[\cdot_\lambda\cdot]^{'}$ be two Leibniz conformal algebra structures on $E$ both containing $R$ as a Leibniz conformal subalgebra.
If there exists a Leibniz conformal algebra isomorphism $\varphi: (E, [\cdot_\lambda \cdot])\rightarrow (E,[\cdot_\lambda\cdot]^{'})$ which stabilizes
$R$, then $[\cdot_\lambda\cdot]$ and $[\cdot_\lambda\cdot]^{'}$ are called \emph{equivalent}. In this case, we denote it by
$(E,[\cdot_\lambda\cdot])\equiv (E,[\cdot_\lambda\cdot]^{'})$.
If there exists a Leibniz conformal algebra isomorphism $\varphi: (E, [\cdot_\lambda \cdot])\rightarrow (E,[\cdot_\lambda\cdot]^{'})$ which stabilizes
$R$ and co-stabilizes $Q$, then $[\cdot_\lambda\cdot]$ and $[\cdot_\lambda\cdot]^{'}$ are called \emph{cohomologous}. In this case, we denote it by
$(E,[\cdot_\lambda\cdot])\approx (E,[\cdot_\lambda\cdot]^{'})$.
\end{defi}

Obviously, $\equiv $ and $\approx $ are equivalence relations on the set of all Leibniz conformal algebra structures on
$E$ containing $R$ as a Leibniz conformal subalgebra. We denote  the set
of all equivalence classes via $\equiv $ (resp. $\approx $) by $\text{CExtd}(E,R)$ (resp. $\text{CExtd}^{'}(E,R)$). It is easy to see that $\text{CExtd}(E,R)$ is the classifying object of the
$\mathbb{C}[\partial]$-split extending structures problem, and there exists a canonical projection $\text{CExtd}^{'}(E,R)\twoheadrightarrow \text{CExtd}(E,R) $.

\section{Unified products of Leibniz conformal algebras}

In this section, a unified product of Leibniz conformal algebras is defined and used to provide a theoretical answer to the $\mathbb{C}[\partial]$-split extending structures problem for Leibniz conformal algebras.% using this tool.

\begin{defi}
Let $(R,[\cdot_\lambda\cdot ])$ be a Leibniz conformal algebra  and $Q$ a $\mathbb{C}[\partial]$-module. An \emph{extending datum} of $R$ by $Q$ is
a system $\Omega(R, Q)=(\leftharpoonup_\lambda, \rightharpoonup_\lambda, \lhd_\lambda, \rhd_\lambda, $$ f_\lambda, $ $\{\cdot_\lambda \cdot\})$ consisting of six conformal bilinear maps %given by
\begin{eqnarray*}
&\leftharpoonup_\lambda: R\times Q\rightarrow R[\lambda],&~~~\rightharpoonup_\lambda:
R\times Q\rightarrow Q[\lambda],\\
&\lhd_\lambda: Q\times R\rightarrow Q[\lambda],&~~~\rhd_\lambda: Q\times R\rightarrow R[\lambda],\\
&f_\lambda: Q\times Q \rightarrow R[\lambda],&~~~\{\cdot_\lambda\cdot\}: Q\times Q \rightarrow Q[\lambda].
\end{eqnarray*}
Let $\Omega(R, Q)=(\leftharpoonup_\lambda, \rightharpoonup_\lambda,\lhd_\lambda, \rhd_\lambda,  f_\lambda, \{\cdot_\lambda \cdot\} )$ be an extending datum. Denote by $R\natural_{\Omega(R, Q)}$ $Q$ $=R\natural Q$ the $\mathbb{C}[\partial]$-module
$R\times Q$ with the natural $\mathbb{C}[\partial]$-module action: $\partial (a, x)=(\partial a, \partial x)$ and the bilinear map $
[\cdot_\lambda \cdot]: (R\times Q)\times (R\times Q)\rightarrow (R\times Q)[\lambda]$ defined by
\begin{gather}
[(a,x)_\lambda (b,y)]\nonumber\\
\label{e3}=([a_\lambda b]+a\leftharpoonup_\lambda y+x\rhd_\lambda b+f_\lambda(x,y),a\rightharpoonup_\lambda y+x\lhd_\lambda b+\{x_\lambda y\}),
\end{gather}
for all $a$, $b\in R$, $x$, $y\in Q$. Since $\lhd_\lambda$, $\rhd_\lambda$, $\leftharpoonup_\lambda$, $\rightharpoonup_\lambda$, $f_\lambda$ and $\{\cdot_\lambda\cdot\}$ are conformal bilinear maps, the $\lambda$-product defined by (\ref{e3}) satisfies conformal sesquilinearity. Then
$R\natural Q$ is called the \emph{unified product} of $R$ and $Q$ associated with $\Omega(R, Q)$ if it is a Leibniz conformal algebra with the $\lambda$-product given by (\ref{e3}). In this case, the extending datum $\Omega(R, Q)$ is called a \emph{Leibniz conformal extending structure} of $R$ by $Q$.
We denote by $\mathcal{TCL}(R,Q)$ the set of all Leibniz conformal extending structures of $R$ by $Q$.
\end{defi}

By (\ref{e3}), the following equalities hold in $R\natural Q$ for all $a$, $b\in R$, $x$, $y\in Q$:
\begin{eqnarray}
&&\label{er1}[(a,0)_{\lambda}(b,0)]=([a_{\lambda} b],0),~~~[(a,0)_{\lambda}(0,y)]=(a\leftharpoonup_\lambda y, a\rightharpoonup_\lambda y),\\
&&\label{er2}[(0,x)_{\lambda}(b,0)]=(x\rhd_{\lambda} b, x\lhd_{\lambda}b),~~~[(0,x)_{\lambda} (0,y)]=(f_{\lambda}(x,y),\{x_{\lambda} y\}).
\end{eqnarray}

\begin{thm}\label{t1}
Let $R$ be a Leibniz conformal algebra, $Q$ be a $\mathbb{C}[\partial]$-module and $\Omega(R, Q)$ an extending datum of $R$ by $Q$.
Then $R\natural Q$ is a Leibniz conformal algebra if and only if the following conditions are satisfied for all
$a$, $b\in R$ and $x$, $y$, $z\in Q$:
\begin{eqnarray*}
(L1)&&~~[a_\lambda b]\rightharpoonup_{\lambda+\mu}x=a\rightharpoonup_\lambda(b\rightharpoonup_\mu x)-b\rightharpoonup_\mu(a\rightharpoonup_\lambda x),\\
(L2)&&~~[a_\lambda b]\leftharpoonup_{\lambda+\mu}x=[a_\lambda(b\leftharpoonup_\mu x)]+a\leftharpoonup_\lambda(b\rightharpoonup_\mu x)-[b_\mu(a\leftharpoonup_\lambda x)]-b\leftharpoonup_\mu(a\rightharpoonup_\lambda x),\\
(L3)&&~~x\lhd_\mu [a_\lambda b]=a\rightharpoonup_\lambda (x\lhd_\mu b)-(a\rightharpoonup_\lambda x)\lhd_{\lambda+\mu} b,\\
(L4) &&~~x\rhd_\mu [a_\lambda b]=[a_\lambda (x\rhd_\mu b)]+a\leftharpoonup_\lambda (x\lhd_\mu b)-[(a\leftharpoonup_\lambda x)_{\lambda+\mu}b]-(a\rightharpoonup_\lambda x)\rhd_{\lambda+\mu} b,\\
(L5) &&~~[(x\rhd_\mu a)_{\lambda+\mu}b]+(x\lhd_\mu a)\rhd_{\lambda+\mu}b+[(a\leftharpoonup_\lambda x)_{\lambda+\mu}b]
+(a\rightharpoonup_\lambda x)\rhd_{\lambda+\mu}b=0,\\
(L6) &&~~(x\lhd_\mu a)\lhd_{\lambda+\mu} b+(a\rightharpoonup_\lambda x)\lhd_{\lambda+\mu}b=0,\\
(L7)&&~~a\rightharpoonup_\lambda\{x_\mu y\}=\{(a\rightharpoonup_\lambda x)_{\lambda+\mu}y\}+x\lhd_\mu(a\leftharpoonup_\lambda y)+\{x_\mu(a\rightharpoonup_\lambda y)\}\\&&\ \ \ \ \ \ \ \ \ \ \ \ \ \ \ \ \ \ \ \,+(a\leftharpoonup_\lambda x)\rightharpoonup_{\lambda+\mu}y,\\
(L8)&&~~a\leftharpoonup_\lambda\{x_\mu y\}=(a\leftharpoonup_\lambda x)\leftharpoonup_{\lambda+\mu}y+f_{\lambda+\mu}(a\rightharpoonup_\lambda x,y)+x\rhd_\mu(a\leftharpoonup_\lambda y)\\&&\ \ \ \ \ \ \ \ \ \ \ \ \ \ \ \ \ \ \
+f_\mu(x,a\rightharpoonup_\lambda y)-[a_\lambda f_\mu(x,y)],\\
(L9)&&~~(x\rhd_\mu a)\leftharpoonup_{\lambda+\mu} y+f_{\lambda+\mu}(x\lhd_\mu a,y)+(a\leftharpoonup_\lambda x)\leftharpoonup_{\lambda+\mu}y+f_{\lambda+\mu}(a\rightharpoonup_\lambda x,y)=0,\\
(L10)&&~~\{(a\rightharpoonup_\lambda x)_{\lambda+\mu}y\}+(a\leftharpoonup_\lambda x)\rightharpoonup_{\lambda+\mu}y
+(x\rhd_\mu a)\rightharpoonup_{\lambda+\mu}y+\{(x\lhd_\mu a)_{\lambda+\mu}y\}=0,\\
(L11)&&~~\{x_\lambda y\}\rhd_{\lambda+\mu}a=x\rhd_\lambda(y\rhd_\mu a)+f_\lambda(x,y\lhd_\mu a)
-[f_\lambda(x,y)_{\lambda+\mu}a]-y\rhd_\mu(x\rhd_\lambda a)\\&&\ \ \ \ \ \ \ \ \ \ \ \ \ \ \ \ \ \ \ \ \,-f_\mu(y,x\lhd_\lambda a),
\end{eqnarray*}
\begin{eqnarray*}
(L12)&&~~\{x_\lambda y\}\lhd_{\lambda+\mu}a=x\lhd_\lambda(y\rhd_\mu a)+\{x_\lambda(y\lhd_\mu a)\}-y\lhd_\mu(x\rhd_\lambda a)-\{y_\mu(x\lhd_\lambda a)\},\\
(L13)&&~~x\rhd_\lambda f_\mu(y,z)=f_\lambda(x,y)\leftharpoonup_{\lambda+\mu}z+f_{\lambda+\mu}(\{x_\lambda y\},z)
+y\rhd_\mu f_\lambda(x,z)\\&&\ \ \ \ \ \ \ \ \ \ \ \ \ \ \ \ \ \ \ \,+f_\mu(y,\{x_\lambda z\})-f_\lambda(x,\{y_\mu z\}),\\
(L14)&&x\lhd_\lambda f_\mu(y,z)=f_\lambda(x,y)\rightharpoonup_{\lambda+\mu}z+y\lhd_\mu f_\lambda(x,z)+\{\{x_\lambda y\}_{\lambda+\mu}z\}+\{y_\mu\{x_\lambda z\}\}\\&&\ \ \ \ \ \ \ \ \ \ \ \ \ \ \ \ \ \ \ \,-\{x_\lambda\{y_\mu z\}\}.
\end{eqnarray*}

\end{thm}
\begin{proof}
Set
$$J((a,x),(b,y),(c,z))=[(a,x)_\lambda[(b,y)_\mu (c,z)]]-[[(a,x)_\lambda (b,y)]_{\lambda+\mu}(c,z)]-[(b,y)_\mu[(a,x)_\lambda (c,z)]]$$
for all $a$, $b$, $c\in R$, $x$, $y$, $z\in Q.$ It is easy to see that Jacobi identity  holds for \eqref{e3} if and only if
$J((a,0),(b,0),$ $(c,0))=0$,
$J((a,0),(b,0),(0,x))=0$, $J((a,0),(0,x),(b,0))=0$, $J((0,x),(a,0),$ $(b,0))=0$,
$J((a,0),(0,x),(0,y))=0$, $J((0,x),(a,0),(0,y))=0$, $J((0,x),(0,y),$ $(a,0))$$=0$,
and $J((0,x),(0,y),(0,z))=0$
 for all $a$, $b$, $c\in R$ and  $x$, $y$, $z\in Q$.

Obviously, $J((a,0),(b,0),(c,0))=0$ if and only if  $R$ is a Leibniz conformal algebra. Since
\begin{eqnarray*}
J((a,0),(0,x),(b,0))&=&[(a,0)_\lambda(x\rhd_\mu b,x\lhd_\mu b)]-[(a\leftharpoonup_\lambda x, a\rightharpoonup_\lambda x)_{\lambda+\mu}(b,0)]\\&&
-[(0,x)_\mu([a_\lambda b],0)]\\
&=&([a_\lambda(x\rhd_\mu b)]+a\leftharpoonup_\lambda(x\lhd_\mu b),a\rightharpoonup_\lambda(x\lhd_\mu b))\\&&
-([(a\leftharpoonup_\lambda x)_{\lambda+\mu}b]+(a\rightharpoonup_\lambda x)\rhd_{\lambda+\mu}b,(a\rightharpoonup_\lambda x)\lhd_{\lambda+\mu}b)\\
&&-(x\rhd_\mu [a_\lambda b], x\lhd_\mu[a_\lambda b])\\&=&0,
\end{eqnarray*}
we have $J((a,0),(0,x),(b,0))=0$ if and only if $(L3)$ and $(L4)$ are satisfied.

With a similar computation, one can easily check that the Leibniz identity  is satisfied for (\ref{e3}) if and only if $(L1)$-$(L14)$ hold.
\end{proof}

\begin{rmk}
By (L1), (L3) and (L6), %it is easy to see that
$(Q,\rightharpoonup_\lambda,\lhd_\lambda)$ is a module of $R$.
\end{rmk}

\begin{ex}
Suppose that $\Omega(R, Q)=(\leftharpoonup_\lambda, \rightharpoonup_\lambda, \lhd_\lambda, \rhd_\lambda, $$ f_\lambda, $ $\{\cdot_\lambda \cdot\})$ is an extending datum of a Leibniz conformal algebra $R$ by a $\mathbb{C}[\partial]$-module $Q$ with $\leftharpoonup_\lambda$, $\rhd_\lambda$, $ f_\lambda, $ and $\{\cdot_\lambda \cdot\}$ trivial. We denote this extending datum simply by $(\rightharpoonup_\lambda, \lhd_\lambda)$. Obviously, this extending datum is a Leibniz conformal extending structure if and only if $(Q,\rightharpoonup_\lambda,\lhd_\lambda)$ is a module of $R$. The corresponding Leibniz conformal algebra on $R\oplus Q$ is just the semi-direct sum of $R$ and $Q$.
\end{ex}
\begin{thm}\label{t2}
Let $R$ be a Leibniz conformal algebra and $Q$ a $\mathbb{C}[\partial]$-module. Set $E=R\oplus Q$ where the direct sum is the sum of $\mathbb{C}[\partial]$-modules. Suppose that $E$ has a Leibniz conformal algebra structure $[\cdot_\lambda \cdot]$  such that $R$ is a subalgebra. Then
there exists a Leibniz conformal extending structure $\Omega(R,Q)=(\leftharpoonup_\lambda, \rightharpoonup_\lambda,\lhd_\lambda,\rhd_\lambda, f_\lambda, \{\cdot_\lambda\cdot\} )$ of
$R$ by $Q$ and an isomorphism of Leibniz conformal algebras $E\approx R\natural Q$.
\end{thm}
\begin{proof}
Since $E=R\oplus Q$, there is a natural $\mathbb{C}[\partial]$-module homomorphism $p: E\rightarrow R$ such that
$p(a)=a$ for all $a\in R$. Then we can define an extending datum $\Omega(R,Q)=(\leftharpoonup_\lambda, \rightharpoonup_\lambda, \lhd_\lambda,\rhd_\lambda, f_\lambda, \{\cdot_\lambda\cdot\} )$ of $R$ by $Q$ as follows ($a\in R, x,y\in Q$):
\begin{eqnarray}
&\leftharpoonup_\lambda:R \times Q\rightarrow R[\lambda],&~~~~~~~a\leftharpoonup_\lambda x=p([a_\lambda x]),\\
&\rightharpoonup_\lambda: R \times Q\rightarrow Q[\lambda],&~~~~~~~a\rightharpoonup_\lambda x=[a_\lambda x]-p([a_\lambda x]),\\
&\rhd_\lambda: Q\times R\rightarrow R[\lambda],&~~~~~~~~x\rhd_\lambda a=p([x_\lambda a]),\\
&\lhd_\lambda: Q\times R\rightarrow Q[\lambda],&~~~~~~~~x\lhd_\lambda a=[x_\lambda a]-p([x_\lambda a]),\\
&f_\lambda: Q\times Q\rightarrow R[\lambda],&~~~~~~~~f_\lambda(x,y)=p([x_\lambda y]),\\
&\{\cdot_\lambda \cdot\}: Q\times Q\rightarrow Q[\lambda],&~~~~\{x_\lambda y\}=[x_\lambda y]-p([x_\lambda y]).
\end{eqnarray}
With the similar proof as that in \cite[Theorem 3.5]{HS}, it is easy to show that $\Omega(R,Q)=(\leftharpoonup_\lambda, \rightharpoonup_\lambda,\lhd_\lambda,\rhd_\lambda, f_\lambda, \{\cdot_\lambda\cdot\} )$ is a Leibniz conformal extending structure and
$E\approx R\natural Q$ as Leibniz conformal algebras.
\end{proof}
\begin{defi}\label{DD1}
Let $R$ be a Leibniz conformal algebra and $Q$ a $\mathbb{C}[\partial]$-module. If there exists a pair of $\mathbb{C}[\partial]$-module homomorphisms $(u,v)$, where $u: Q\rightarrow R$, $v\in \text{Aut}_{\mathbb{C}[\partial]}(Q)$ such that
a Leibniz conformal extending structure $\Omega(R,Q)=(\leftharpoonup_\lambda,\rightharpoonup_\lambda, \lhd_\lambda,\rhd_\lambda,
f_\lambda,\{\cdot_\lambda\cdot\} )$ can be obtained from
another one $\Omega^{'}(R,Q)=(\leftharpoonup_\lambda^{'},\rightharpoonup_\lambda^{'}, \lhd_\lambda^{'},\rhd_\lambda^{'},
f_\lambda^{'}, \{\cdot_\lambda\cdot\} ^{'} )$ via
$(u,v)$ as follows:
\begin{gather}
\label{g1}a\rightharpoonup_\lambda x=v^{-1}(a\rightharpoonup_\lambda^{'}v(x)),\\
a\leftharpoonup_\lambda x=[a_\lambda u(x)]+a\leftharpoonup_\lambda^{'}v(x)-u(a\rightharpoonup_\lambda x),\\
x\lhd_\lambda a=v^{-1}(v(x)\lhd_\lambda^{'}a),\\
x\rhd_\lambda a=[u(x)_\lambda a]+v(x)\rhd_\lambda^{'}a-u(x\lhd_\lambda a),\\
\{x_\lambda y\}=v^{-1}(u(x)\rightharpoonup_\lambda^{'}v(y))+v^{-1}(v(x)\lhd_\lambda^{'}u(y))+v^{-1}(\{v(x)_\lambda v(y)\}^{'}),\\
\label{g6}f_\lambda(x,y)=[u(x)_\lambda u(y)]+u(x)\leftharpoonup_\lambda^{'}v(y)
+v(x)\rhd_\lambda^{'}u(y)+f_\lambda^{'}(v(x),v(y))-u(\{x_\lambda y\}),
\end{gather}
for all $a\in R$, $x$, $y\in Q$, then $\Omega(R,Q)$ and  $\Omega^{'}(R,Q)$ are called \emph{equivalent}. In this case, we denote it
by $\Omega(R,Q)\equiv \Omega^{'}(R,Q)$.

In particular, if $v=Id_Q$, then $\Omega(R,Q)$ and  $\Omega^{'}(R,Q)$ are called \emph{cohomologous}. In this case, we denote it by $\Omega(R,Q)\approx \Omega^{'}(R,Q)$.
\end{defi}

\begin{lem}\label{l1}
Suppose that $\Omega(R,Q)=(\leftharpoonup_\lambda,\rightharpoonup_\lambda, \lhd_\lambda,\rhd_\lambda,
f_\lambda,\{\cdot_\lambda\cdot\} )$ and $\Omega^{'}(R,Q)=($ $\leftharpoonup_\lambda^{'}$ $\rightharpoonup_\lambda^{'}$, $\lhd_\lambda^{'},\rhd_\lambda^{'},$ $f_\lambda^{'},$ $ \{\cdot_\lambda\cdot\}^{'})$
are two Leibniz conformal extending structures of $R$ by $Q$.
Let $R\natural Q$ and $R\natural^{'} Q$ be the corresponding unified products.
Then $R\natural Q \equiv  R\natural^{'} Q$ if and only if $\Omega(R,Q)\equiv \Omega^{'}(R,Q)$,
and $R\natural Q \approx  R\natural^{'} Q$ if and only if $\Omega(R,Q)\approx \Omega^{'}(R,Q)$.
\end{lem}
\begin{proof}
Let $\varphi: R\natural Q\rightarrow R\natural^{'} Q$ be an isomorphism of Leibniz conformal algebras which stabilizes $R$.
According to that $\varphi$ stabilizes $R$,  $\varphi(a,0)=(a,0)$. Then we can assume
$\varphi(a,x)=(a+u(x),v(x))$ where $u:Q\rightarrow R$, $v:Q\rightarrow Q$ are two linear maps.
Obviously, $\varphi$ is a $\mathbb{C}[\partial]$-module homomorphism if and only if
$u$, $v$ are two $\mathbb{C}[\partial]$-module homomorphisms.
Similar to that in \cite[Lemma 3.6]{HS}, it is easy to check that
$\varphi$ is an
algebra isomorphism which stabilizes $R$ if and only if $v$ is a $\mathbb{C}[\partial]$-module isomorphism, and (\ref{g1})-(\ref{g6}) hold.
Moreover, it is obvious that a Leibiniz conformal algebra isomorphism $\varphi$ stabilizes $R$ and costabilizes
$Q$ if and only if $v=Id_Q$.

Then the results follow from Definition \ref{DD1}.% and the above discussions.
\end{proof}

Now we can give an answer to the $\mathbb{C}[\partial]$-split extending structures problem for Leibniz conformal algebras.
\begin{thm}\label{t4}
Let $R$ be a Leibniz conformal algebra and $Q$ a $\mathbb{C}[\partial]$-module. Set $E=R\oplus Q$ where the direct sum is the sum of $\mathbb{C}[\partial]$-modules. %Then we get\\
\begin{itemize}
\item[(1)]Denote $\mathcal{LH}_{R}^2(Q,R):=\mathcal{TCL}(R,Q)/\equiv$. Then the map
\begin{eqnarray}
\mathcal{LH}_{R}^2(Q,R)\rightarrow \text{CExtd}(E,R),~~~~\overline{\Omega(R,Q)}\mapsto (R\natural Q,[\cdot_\lambda \cdot])
\end{eqnarray}
is bijective, where $\overline{\Omega(R,Q)}$ is the equivalence class of $\Omega(R,Q)$ under $\equiv$.\\
\item[(2)] Denote $\mathcal{LH}^2(Q,R):=\mathcal{TCL}(R,Q)/\approx$. Then the map
\begin{eqnarray}
\mathcal{LH}^2(Q,R)\rightarrow \text{CExtd}^{'}(E,R),~~~~\overline{\overline{\Omega(R,Q)}}\mapsto (R\natural Q, [\cdot_\lambda \cdot])
\end{eqnarray}
is bijective, where $\overline{\overline{\Omega(R,Q)}}$ is the equivalence class of $\Omega(R,Q)$ under $\approx$.
\end{itemize}
\end{thm}

\begin{proof}
It can be directly obtained from Theorems \ref{t1},  \ref{t2} and Lemma \ref{l1}.
\end{proof}

\section{Unified product when $Q=\mathbb{C}[\partial]x$}
In this section, we will apply the general theory developed in Section 3 to the special case when
$R$ and $Q$ are free $\mathbb{C}[\partial]$-modules with $Q=\mathbb{C}[\partial]x$.
\begin{defi}
Let $R=\mathbb{C}[\partial]V$ be a Leibniz conformal algebra which is free as a $\mathbb{C}[\partial]$-module. A \emph{flag datum} of the first kind of $R$ is
a 5-tuple $(h_\lambda(\cdot,\partial),$ $ D_\lambda,$
$T_\lambda, $ $Q_0(\lambda,\partial), $ $P(\lambda,\partial))$ where $P(\lambda,\partial)\in \mathbb{C}[\lambda,\partial]$,
$Q_0(\lambda,\partial)\in R[\lambda]$, $h_\lambda(\cdot,\partial):R\rightarrow \mathbb{C}[\lambda,\partial]$ and $D_\lambda: R\rightarrow R[\lambda]$ are two left conformal linear maps,  and  $T_\lambda: R\rightarrow R[\lambda]$ is a conformal linear map
satisfying the following conditions $(a, b\in V)$:
\begin{gather}
\label{fg1}h_{\lambda+\mu}([a_\lambda b],\partial)=h_\mu(b,\lambda+\partial)h_\lambda(a,\partial)-h_\lambda(a,\mu+\partial)h_\mu(b,\partial),\\
\label{fg2}D_{\lambda+\mu}([a_\lambda b])=[a_\lambda D_\mu(b)]+h_\mu(b,\lambda+\partial)D_\lambda(a)
-[b_\mu D_\lambda(a)]-h_\lambda(a,\mu+\partial)D_\mu(b),\\
\label{fg3}T_\mu([a_\lambda b])=[a_\lambda T_\mu(b)]-[D_\lambda(a)_{\lambda+\mu}b]-h_\lambda(a,-\lambda-\mu)T_{\lambda+\mu}(b),\\
\label{fg4}T_\mu([a_\lambda b])=[a_\lambda T_\mu(b)]+[{T_\mu(a)}_{\lambda+\mu}b],\\
\label{fg5}P(\mu,\lambda+\partial)h_\lambda(a,\partial)=h_\lambda(a,-\lambda-\mu)P(\lambda+\mu,\partial)
+h_\lambda(a,\mu+\partial)P(\mu,\partial)+h_{\lambda+\mu}(D_\lambda(a),\partial),\\
\label{fg6}P(\mu,\lambda+\partial)D_\lambda(a)=D_{\lambda+\mu}(D_\lambda(a))+h_\lambda(a,-\lambda-\mu)Q_0(\lambda+\mu,\partial)
+T_\mu(D_\lambda(a))\\
+h_\lambda(a,\mu+\partial)Q_0(\mu,\partial)-[a_\lambda Q_0(\mu,\partial)],\nonumber\\
\label{fg7}D_{\lambda+\mu}(T_\mu(a))+D_{\lambda+\mu}(D_\lambda(a))+h_\lambda(a,-\lambda-\mu)Q_0(\lambda+\mu,\partial)=0,\\
\label{fg8}h_\lambda(a,-\lambda-\mu)P(\lambda+\mu,\partial)+h_{\lambda+\mu}(D_\lambda(a),\partial)+h_{\lambda+\mu}(T_\mu(a),\partial)=0,\\
\label{fg9}P(\lambda,-\lambda-\mu)T_{\lambda+\mu}(a)=T_\lambda(T_\mu(a))-[{Q_0(\lambda,\partial)}_{\lambda+\mu}a]
-T_\mu(T_\lambda(a)),\\
T_\lambda(Q_0(\mu,\partial))=D_{\lambda+\mu}(Q_0(\lambda,\partial))+P(\lambda,-\lambda-\mu)Q_0(\lambda+\mu,\partial)
+T_\mu(Q_0(\lambda,\partial))\nonumber\\
\label{fg10}+P(\lambda,\mu+\partial)Q_0(\mu,\partial)-P(\mu,\lambda+\partial)Q_0(\lambda,\partial),\\
\label{fg11}h_{\lambda+\mu}(Q_0(\lambda,\partial),\partial)+P(\lambda,-\lambda-\mu)P(\lambda+\mu,\partial)
+P(\lambda,\mu+\partial)P(\mu,\partial)-P(\mu,\lambda+\partial)P(\lambda,\partial)=0.
\end{gather}
\end{defi}
%Denote by $\mathcal{F}_1(R)$ the set of all flag datums of the first kind of $R$.
\begin{defi}
Let $R=\mathbb{C}[\partial]V$ be a Leibniz conformal algebra which is free as a $\mathbb{C}[\partial]$-module. A \emph{flag datum} of the second kind of $R$ is
a 5-tuple $(h_\lambda(\cdot,\partial),$ $ D_\lambda,$
$T_\lambda, $ $Q_0(\lambda,\partial), $ $P(\lambda,\partial))$ where $P(\lambda,\partial)\in \mathbb{C}[\lambda,\partial]$,
$Q_0(\lambda,\partial)\in R[\lambda]$, $h_\lambda(\cdot,\partial):R\rightarrow \mathbb{C}[\lambda,\partial]$ and $D_\lambda: R\rightarrow R[\lambda]$ are two left conformal linear maps and $h_\lambda(\cdot,\partial)$ is a non-trivial map, and  $T_\lambda: R\rightarrow R[\lambda]$ is a conformal linear map
satisfying (\ref{fg1}), (\ref{fg2}), (\ref{fg6}), (\ref{fg10}) and the following conditions $(a, b\in V)$:
\begin{gather}
\label{fgg11}T_\mu([a_\lambda b])=[a_\lambda T_\mu(b)]-h_{-\lambda-\mu-\partial}(b,\lambda+\partial)D_\lambda(a)-[D_\lambda(a)_{\lambda+\mu}b]-h_\lambda(a,-\lambda-\mu)T_{\lambda+\mu}(b),\\
\label{fg12}[{T_\mu(a)}_{\lambda+\mu}b]+[D_\lambda(a)_{\lambda+\mu}b]=0,\\
P(\mu,\lambda+\partial)h_\lambda(a,\partial)=h_\lambda(a,-\lambda-\mu)P(\lambda+\mu,\partial)
-h_{-\mu-\partial}(D_\lambda(a),\partial)\nonumber\\
\label{fg13}\ \ \ \ \ \ \ \ \ \ \ \ \ \ \ \ \ \ \ \, +h_\lambda(a,\mu+\partial)P(\mu,\partial)+h_{\lambda+\mu}(D_\lambda(a),\partial),\\
\label{fg14}D_{\lambda+\mu}(T_\mu(a))+D_{\lambda+\mu}(D_\lambda(a))=0,\\
\label{fg15}h_{\lambda+\mu}(D_\lambda(a),\partial)+h_{\lambda+\mu}(T_\mu(a),\partial)=0,\\
P(\lambda,-\lambda-\mu)T_{\lambda+\mu}(a)=T_\lambda(T_\mu(a))-h_{-\lambda-\mu-\partial}(a,\lambda+\partial)Q_0(\lambda,\partial)\nonumber\\
\label{fg16}-[{Q_0(\lambda,\partial)}_{\lambda+\mu}a]
-T_\mu(T_\lambda(a))+h_{-\lambda-\mu-\partial}(a,\mu+\partial)Q_0(\mu,\partial),\\
P(\lambda,-\lambda-\mu)h_{-\lambda-\mu-\partial}(a,\partial)=h_{-\lambda-\partial}(T_\mu(a),\partial)
+h_{-\lambda-\mu-\partial}(a,\lambda+\partial)P(\lambda,\partial)\nonumber\\
\label{fg17}\ \ \ \ \ \ \ \ \ \ \ \ \ \ \ \ \ \ \ \ \ \ \ \ \ \ \ \ \ \ \ \ \ \ \ \ \ \ \,-h_{-\mu-\partial}(T_\lambda(a),\partial)
-h_{-\lambda-\mu-\partial}(a,\mu+\partial)P(\mu,\partial),\\
-h_{-\lambda-\partial}(Q_0(\mu,\partial),\partial)=h_{\lambda+\mu}(Q_0(\lambda,\partial),\partial)-h_{-\mu-\partial}(Q_0(\lambda,\partial),\partial)+P(\lambda,-\lambda-\mu)P(\lambda+\mu,\partial)\nonumber\\
\label{fg18}+P(\lambda,\mu+\partial)P(\mu,\partial)-P(\mu,\lambda+\partial)P(\lambda,\partial).
\end{gather}
\end{defi}

Denote by $\mathcal{F}_1(R)$ (resp. $\mathcal{F}_2(R)$) the set of all flag datums of the first (resp. second) kind of $R$. Set
$\mathcal{F}(R)=\mathcal{F}_1(R)\bigcup \mathcal{F}_2(R)$. The elements in $\mathcal{F}(R)$ are called \emph{flag datums} of $R$.

\begin{pro}\label{pr1}
Let $R=\mathbb{C}[\partial]V$ be a Leibniz conformal algebra which is a free $\mathbb{C}[\partial]$-module and $Q=\mathbb{C}[\partial]x$ be a free $\mathbb{C}[\partial]$-module of rank 1.
Then there is a bijection between the set  $\mathcal{TCL}(R,Q)$ of all Leibniz conformal extending structures of $R$ by $Q$ and $\mathcal{F}(R)=\mathcal{F}_1(R)\bigcup \mathcal{F}_2(R)$.
\end{pro}
\begin{proof}
Let $\Omega(R,Q)=(\leftharpoonup_\lambda,\rightharpoonup_\lambda,\lhd_\lambda,\rhd_\lambda,f_\lambda,
\{\cdot_\lambda\cdot\})$ be a Leibniz conformal extending structure. Since $Q=\mathbb{C}[\partial]x$ is a free
$\mathbb{C}[\partial]$-module of rank 1, we can set
\begin{eqnarray*}
& a\rightharpoonup_\lambda x=h_\lambda(a,\partial)x,&~~a\leftharpoonup_\lambda x=D_\lambda(a),\\
&x\lhd_\lambda a=g_\lambda(a,\partial)x,&~~x\rhd_\lambda a=T_\lambda(a),\\
&\{x_\lambda x\}=P(\lambda,\partial)x,&~~f_\lambda(x,x)=Q_0(\lambda,\partial),
\end{eqnarray*}
where $a\in V$, $P(\lambda,\partial)\in \mathbb{C}[\lambda,\partial]$, $Q_0(\lambda,\partial)\in R[\lambda]$, and $h_\lambda(\cdot,\partial)$, $g_\lambda(\cdot,\partial): R\rightarrow \mathbb{C}[\lambda,\partial]$ and $D_\lambda$, $T_\lambda: R\rightarrow R[\lambda]$ are linear maps.

Since $\leftharpoonup_\lambda$, $\rightharpoonup_\lambda$, $\lhd_\lambda$ and $\rhd_\lambda$ are conformal bilinear maps, we get $h_\lambda(\cdot,\partial)$ and $D_\lambda$ are left conformal linear maps, and $g_\lambda(\cdot,\partial)$ and $T_\lambda$
are two conformal linear maps. By $(L6)$, we get  $(g_\mu(a,-\lambda-\mu)+h_\lambda(a,-\lambda-\mu))g_{\lambda+\mu}(b,\partial)=0$. Therefore, there are two cases: $g_\lambda(\cdot,\partial)=0$ and $g_\mu(a,-\lambda-\mu)=-h_\lambda(a,-\lambda-\mu)$ with $g_\lambda(\cdot,\partial)$ not a trivial map. When $g_\lambda(\cdot,\partial)=0$, it is easy to check that $(L1)$-$(L14)$ hold if and only if $(\ref{fg1})$-$(\ref{fg11})$ are satisfied, i.e. $(h_\lambda(\cdot,\partial),$ $ D_\lambda,$
$T_\lambda, $ $Q_0(\lambda,\partial), $ $P(\lambda,\partial))\in \mathcal{F}_1(R)$. If $g_\mu(a,-\lambda-\mu)=-h_\lambda(a,-\lambda-\mu)$ when $g_\lambda(\cdot,\partial)$ is not a trivial map, it is also easy to get that $(L1)$-$(L14)$ hold if and only if (\ref{fg1}), (\ref{fg2}), (\ref{fg6}), (\ref{fg10}), (\ref{fgg11})-(\ref{fg18}) are satisfied, i.e. $(h_\lambda(\cdot,\partial),$ $ D_\lambda,$
$T_\lambda, $ $Q_0(\lambda,\partial), $ $P(\lambda,\partial))\in \mathcal{F}_2(R)$.
\end{proof}

Suppose that $(h_\lambda(\cdot,\partial), D_\lambda,
T_\lambda, Q_0(\lambda,\partial), P(\lambda,\partial))\in \mathcal{F}_1(R)$. By Proposition \ref{pr1}, the Leibniz conformal algebra corresponding
to it is the $\mathbb{C}[\partial]$-module
$R\oplus \mathbb{C}[\partial]x$ with the following $\lambda$-products
\begin{eqnarray}
&&[(a,0)_\lambda (b,0)]=([a_\lambda b],0),~~[(0,x)_\lambda(0,x)]=(Q_0(\lambda,\partial),P(\lambda,\partial)x),\\
&&[(a,0)_\lambda(0,x)]=(D_{\lambda}(a),h_{\lambda}(a,\partial)x),~~[(0,x)_\lambda(a,0)]
=(T_\lambda(a),0),
\end{eqnarray}
for any $a$, $b\in V$. We denote this Leibniz conformal algebra by $FC_1\big(R,\mathbb{C}[\partial]x\mid$ $h_\lambda(\cdot,\partial), $ $D_\lambda,$
$T_\lambda,$ $ Q_0(\lambda,\partial),$ $ P(\lambda,\partial)\big)$.

Similarly, if $(h_\lambda(\cdot,\partial), D_\lambda,
T_\lambda, Q_0(\lambda,\partial), P(\lambda,\partial))\in \mathcal{F}_2(R)$, then the corresponding Leibniz conformal algebra is the $\mathbb{C}[\partial]$-module
$R\oplus \mathbb{C}[\partial]x$ with the following $\lambda$-products
\begin{eqnarray}
&&[(a,0)_\lambda (b,0)]=([a_\lambda b],0),~~[(0,x)_\lambda(0,x)]=(Q_0(\lambda,\partial),P(\lambda,\partial)x),\\
&&[(a,0)_\lambda(0,x)]=(D_{\lambda}(a),h_{\lambda}(a,\partial)x),~~[(0,x)_\lambda(a,0)]
=(T_\lambda(a),-h_{-\lambda-\partial}(a,\partial)x),
\end{eqnarray}
for any $a$, $b\in V$. We denote this Leibniz conformal algebra by $FC_2\big(R,\mathbb{C}[\partial]x\mid$ $h_\lambda(\cdot,\partial), $ $D_\lambda,$
$T_\lambda,$ $ Q_0(\lambda,\partial),$ $ P(\lambda,\partial)\big)$.

\begin{thm}\label{t5}
Let $R=\mathbb{C}[\partial]V$ be a Leibniz conformal algebra and $Q=\mathbb{C}[\partial]x$ be a free $\mathbb{C}[\partial]$-module of rank 1. Set $E=R\oplus Q$ as a $\mathbb{C}[\partial]$-module. Then %we have
\begin{itemize}

\item[(1)]$
CExtd(E,R)\cong\mathcal{LH}_R^2(Q,R)\cong (\mathcal{F}_1(R)/\equiv_1)\cup (\mathcal{F}_2(R)/\equiv_2),
$
where $\equiv_1$ is the equivalence relation on %the set
$\mathcal{F}_1(R)$ as follows:
$(h_\lambda(\cdot,\partial), D_\lambda,
T_\lambda, Q_0(\lambda,\partial), P(\lambda,\partial))\equiv_1 (h_\lambda^{'}(\cdot,\partial),$ $D_\lambda^{'},$  $T_\lambda^{'},$ $Q_0^{'}(\lambda,\partial),$ $P^{'}(\lambda,\partial))$
if and only if
$h_\lambda(\cdot,\partial)=h_\lambda^{'}(\cdot,\partial)$ and there exist $u_0\in R$ and $\beta\in \mathbb{C}\backslash\{0\}$ such that (for all $a\in R$)
\begin{gather}
\label{tt1}D_\lambda(a)=[a_\lambda u_0]+\beta D_\lambda^{'}(a)-h_\lambda(a,\partial)u_0,\\
T_\lambda(a)=[{u_0}_\lambda a]+\beta T_\lambda^{'}(a),\\
P(\lambda,\partial)=h^{'}_\lambda(u_0,\partial)+\beta P^{'}(\lambda,\partial),\\
\label{tt2}Q_0(\lambda,\partial)=[{u_0}_\lambda u_0]+\beta D_\lambda^{'}(u_0)
+\beta T_\lambda^{'}(u_0)+\beta^2Q_0^{'}(\lambda,\partial)-P(\lambda,\partial)u_0,
\end{gather}
and
$\equiv_2$ is the equivalence relation on
$\mathcal{F}_2(R)$ as follows:
$(h_\lambda(\cdot,\partial), D_\lambda,
T_\lambda, Q_0(\lambda,\partial), P(\lambda,\partial))$ $\equiv_2 (h_\lambda^{'}(\cdot,\partial),$ $D_\lambda^{'},$  $T_\lambda^{'},$ $Q_0^{'}(\lambda,\partial),$ $P^{'}(\lambda,\partial))$
if and only if
$h_\lambda(\cdot,\partial)=h_\lambda^{'}(\cdot,\partial)$ and there exist $u_0\in R$ and $\beta\in \mathbb{C}\backslash\{0\}$ such that (for all $a\in R$)
\begin{gather}
\label{ttt1}D_\lambda(a)=[a_\lambda u_0]+\beta D_\lambda^{'}(a)-h_\lambda(a,\partial)u_0,\\
T_\lambda(a)=[{u_0}_\lambda a]+\beta T_\lambda^{'}(a)+h_{-\lambda-\partial}(a,\partial)u_0,\\
P(\lambda,\partial)=h^{'}_\lambda(u_0,\partial)-h_{-\lambda-\partial}^{'}(u_0,\partial)+\beta P^{'}(\lambda,\partial),\\
\label{ttt2}Q_0(\lambda,\partial)=[{u_0}_\lambda u_0]+\beta D_\lambda^{'}(u_0)
+\beta T_\lambda^{'}(u_0)+\beta^2Q_0^{'}(\lambda,\partial)-P(\lambda,\partial)u_0.
\end{gather}
The bijection between $(\mathcal{F}_1(R)/\equiv_1)\cup (\mathcal{F}_2(R)/\equiv_2)$ and
$CExtd(E,R)$ is given by
%\begin{eqnarray*}

$\overline{(h_\lambda(\cdot,\partial), D_\lambda,
T_\lambda, Q_0(\lambda,\partial), P(\lambda,\partial))}^{1}\mapsto FC_1(R,\mathbb{C}[\partial]x|h_\lambda(\cdot,\partial), D_\lambda,
T_\lambda, Q_0(\lambda,\partial), P(\lambda,\partial)),$
%\end{eqnarray*}
and

%\begin{eqnarray*}
$\overline{(h_\lambda(\cdot,\partial), D_\lambda,
T_\lambda, Q_0(\lambda,\partial), P(\lambda,\partial))}^{2}\mapsto FC_2(R,\mathbb{C}[\partial]x|h_\lambda(\cdot,\partial), D_\lambda,
T_\lambda, Q_0(\lambda,\partial), P(\lambda,\partial)),$
%\end{eqnarray*}
where $\overline{(h_\lambda(\cdot,\partial), D_\lambda,
T_\lambda, Q_0(\lambda,\partial), P(\lambda,\partial))}^{i}$ is the equivalence class via the relation $\equiv_i$ for $i\in \{1,2\}$.

\item[(2)] $CExtd^{'}(E,R)\cong\mathcal{LH}^2(Q,R)\cong (\mathcal{F}_1(R)/\approx_1)\cup (\mathcal{F}_2(R)/\approx_2)$,
where $\approx_i$ is the equivalence relation on the set
$\mathcal{F}_i(R)$ as follows:
$(h_\lambda(\cdot,\partial), D_\lambda,
T_\lambda, Q_0(\lambda,\partial), P(\lambda,\partial))\approx_1 (h_\lambda^{'}(\cdot,\partial),$ $D_\lambda^{'},$ $T_\lambda^{'},$ $Q_0^{'}(\lambda,\partial),$ $P^{'}(\lambda,\partial))$ if and only if
$h_\lambda(\cdot,\partial)=h_\lambda^{'}(\cdot,\partial)$, and there exists $u_0\in R$  such that (\ref{tt1})-(\ref{tt2}) hold for $\beta=1$, and $(h_\lambda(\cdot,\partial), D_\lambda,
T_\lambda, Q_0(\lambda,\partial), P(\lambda,\partial))\approx_2 (h_\lambda^{'}(\cdot,\partial),$ $D_\lambda^{'},$ $T_\lambda^{'},$ $Q_0^{'}(\lambda,\partial),$ $P^{'}(\lambda,\partial))$ if and only if
$h_\lambda(\cdot,\partial)=h_\lambda^{'}(\cdot,\partial)$, and there exists $u_0\in R$  such that (\ref{ttt1})-(\ref{ttt2}) hold for $\beta=1$.
The bijection between $(\mathcal{F}_1(R)/\approx_1)\cup (\mathcal{F}_2(R)/\approx_2)$ and
$CExtd^{'}(E,R)$ is given by

%\begin{eqnarray*}
$\overline{\overline{(h_\lambda(\cdot,\partial), D_\lambda,
T_\lambda, Q_0(\lambda,\partial), P(\lambda,\partial))}}^{1}\mapsto
FC_1(R,\mathbb{C}[\partial]x|h_\lambda(\cdot,\partial), D_\lambda,
T_\lambda, Q_0(\lambda,\partial), P(\lambda,\partial)),$
%\end{eqnarray*}
and

%\begin{eqnarray*}
$\overline{\overline{(h_\lambda(\cdot,\partial), D_\lambda,
T_\lambda, Q_0(\lambda,\partial), P(\lambda,\partial))}}^{2}\mapsto
FC_2(R,\mathbb{C}[\partial]x|h_\lambda(\cdot,\partial), D_\lambda,
T_\lambda, Q_0(\lambda,\partial), P(\lambda,\partial)),$
%\end{eqnarray*}
where $\overline{(h_\lambda(\cdot,\partial), D_\lambda,
T_\lambda, Q_0(\lambda,\partial), P(\lambda,\partial))}^{i}$ is the equivalence class via the relation $\approx_i$ for $i\in \{1,2\}$.
\end{itemize}
\end{thm}

\begin{proof}
Since %in Lemma \ref{l1},
$u: Q\rightarrow R$ is a $\mathbb{C}[\partial]$-module
homomorphism and $v\in \text{Aut}_{\mathbb{C}[\partial]}(Q)$,
we set $u(x)=u_0$ and $v(x)=\beta x$ where $u_0\in R$ and $\beta\in \mathbb{C}\backslash\{0\}$. Then this theorem can be directly obtained from Lemma \ref{l1}, Theorem \ref{t4} and
Proposition \ref{pr1}.
\end{proof}

\section{Special cases of unified products and examples}

In this section, three special cases of unified products, namely, twisted products, crossed products and bicrossed products are given, and explicit examples are presented.

\subsection{Twisted products}

Let $\Omega(R, Q)=(\leftharpoonup_\lambda, \rightharpoonup_\lambda, \lhd_\lambda, \rhd_\lambda, $$ f_\lambda, $ $\{\cdot_\lambda \cdot\})$ be an extending datum of a Leibniz conformal algebra $R$ by a $\mathbb{C}[\partial]$-module $Q$ with $\leftharpoonup_\lambda$, $\rightharpoonup_\lambda$, $\lhd_\lambda$, and $\rhd_\lambda$ trivial. We  denote this extending structure simply by $\Omega(R, Q)=(f_\lambda, \{\cdot_\lambda \cdot\})$. By Theorem \ref{t1}, $\Omega(R, Q)=(f_\lambda, \{\cdot_\lambda \cdot\})$ is a
Leibniz extending structure if and only if $(Q,\{\cdot_\lambda \cdot\})$ is a Leibniz conformal algebra and
$f_\lambda$ satisfies
\begin{eqnarray}
\label{ts1}[a_\lambda f_\mu(x,y)]=0,~~~~[{f_\lambda(x,y)}_{\lambda+\mu}a]=0,~~~~~~\\
\label{ts2}f_{\lambda+\mu}(\{x_\lambda y\},z)+f_\mu(y,\{x_\lambda z\})-f_\lambda(x,\{y_\mu z\})=0.
\end{eqnarray}
We denote the corresponding unified product by $R\natural_f Q$ and call it the \emph{twisted product} of $R$ and $Q$. The $\lambda$-product on $R\natural_f Q$ is given by
\begin{eqnarray}
[(a,x)_\lambda (b,y)]=([a_\lambda b]+f_\lambda(x,y),\{x_\lambda y\}), \ \ \forall \ a,b\in R, x, y\in Q.
\end{eqnarray}

According to (\ref{ts1}) and (\ref{ts2}), it is easy to get the following.% proposition.
\begin{pro}
Let $R$ be a Leibniz conformal algebra with the trivial left-center or trivial right-center. Then for any Leibniz conformal algebra
$Q$, the twisted product $R\natural_f Q$ is trivial.
\end{pro}

\subsection{Crossed products}

Let $\Omega(R, Q)=(\leftharpoonup_\lambda, \rightharpoonup_\lambda, \lhd_\lambda, \rhd_\lambda, $$ f_\lambda, $ $\{\cdot_\lambda \cdot\})$ be an extending datum of a Leibniz conformal algebra $R$ by a $\mathbb{C}[\partial]$-module $Q$ with  $\rightharpoonup_\lambda$ and $\lhd_\lambda$ trivial. We denote this extending structure simply by $\Omega(R, Q)=(\leftharpoonup_\lambda, \rhd_\lambda, f_\lambda,  \{\cdot_\lambda \cdot\})$. By Theorem \ref{t1}, $\Omega(R, Q)=(\leftharpoonup_\lambda, \rhd_\lambda, f_\lambda,  \{\cdot_\lambda \cdot\})$ is a
Leibniz extending structure if and only if $(Q,\{\cdot_\lambda \cdot\})$ is a Leibniz conformal algebra and the following hold
%they satisfy
\begin{eqnarray*}
(C1)&&~~[a_\lambda b]\leftharpoonup_{\lambda+\mu}x=[a_\lambda(b\leftharpoonup_\mu x)]-[b_\mu(a\leftharpoonup_\lambda x)],\\
(C2) &&~~x\rhd_\mu [a_\lambda b]=[a_\lambda (x\rhd_\mu b)]-[(a\leftharpoonup_\lambda x)_{\lambda+\mu}b],\\
(C3) &&~~[(x\rhd_\mu a)_{\lambda+\mu}b]+[(a\leftharpoonup_\lambda x)_{\lambda+\mu}b]=0,\\
(C4)&&~~a\leftharpoonup_\lambda\{x_\mu y\}=(a\leftharpoonup_\lambda x)\leftharpoonup_{\lambda+\mu}y+x\rhd_\mu(a\leftharpoonup_\lambda y)
-[a_\lambda f_\mu(x,y)],\\
(C5)&&~~(x\rhd_\mu a)\leftharpoonup_{\lambda+\mu} y+(a\leftharpoonup_\lambda x)\leftharpoonup_{\lambda+\mu}y=0,\\
(C6)&&~~\{x_\lambda y\}\rhd_{\lambda+\mu}a=x\rhd_\lambda(y\rhd_\mu a)
-[f_\lambda(x,y)_{\lambda+\mu}a]-y\rhd_\mu(x\rhd_\lambda a),\\
(C7)&&~~x\rhd_\lambda f_\mu(y,z)=f_\lambda(x,y)\leftharpoonup_{\lambda+\mu}z+f_{\lambda+\mu}(\{x_\lambda y\},z)
+y\rhd_\mu f_\lambda(x,z)\\
&&+f_\mu(y,\{x_\lambda z\})-f_\lambda(x,\{y_\mu z\}).\nonumber
\end{eqnarray*}
The unified product associated with $\Omega(R, Q)=(\leftharpoonup_\lambda, \rhd_\lambda, f_\lambda,  \{\cdot_\lambda \cdot\})$ is denoted by $R\natural_{\rhd,\leftharpoonup}^{f}Q$ and called
the \emph{crossed product} of $R$ and $Q$. The $\lambda$-product on $R\natural_{\rhd,\leftharpoonup}^{f}Q$ is given by
\begin{eqnarray}
[(a,x)_\lambda (b,y)]=([a_\lambda b]+x\rhd_\lambda b+a\leftharpoonup_\lambda y+f_\lambda(x,y),\{x_\lambda y\}),
\end{eqnarray}
for any $a$, $b\in R$ and $x$, $y\in Q$. Obviously, $R$ is a bi-sided ideal of $R\natural_{\rhd,\leftharpoonup}^{f}Q$.

By Theorem \ref{t2}, we have
\begin{pro}
Let $R$ be a Leibniz conformal algebra and $Q$ a $\mathbb{C}[\partial]$-module. Set
$E=R\oplus Q$ where the direct sum is the sum of $\mathbb{C}[\partial]$-modules. Suppose
that $E$ is  a Leibniz conformal algebra such that $R$ is a bi-sided ideal of $E$. Then $E$ is isomorphic to
a crossed product $R\natural_{\rhd,\leftharpoonup}^{f}Q$.
\end{pro}

 We have shown that unified product is an effective tool to solve the $\mathbb{C}[\partial]$-split extending structures problem.
As a special case of unified product, crossed product can be used to answer the following problem which is also a special case of $\mathbb{C}[\partial]$-split extending structures problem.

{\noindent\bf Problem 1.} \emph{Given two Leibniz conformal algebras $R$ and $Q$. Set $E = R\oplus Q $ where
the direct sum is the sum of $\mathbb{C}[\partial]$-modules. Describe and classify all Leibniz conformal
algebra structures on $E$ such that $R$ is a bi-sided ideal of $E$ up to isomorphisms which stablize $R$.}

By the discussions in Section 3, this problem can be solved by the cohomological type object $\mathcal{LH}_R^2(Q,R)=\mathfrak{TCL}(R,Q)/\equiv$, where $\lhd_\lambda$ and $\rightharpoonup_\lambda$ are trivial in these Leibniz conformal extending structures. For convenience, we denote this
cohomological type object by $\mathcal{LHC}_R^2(Q,R)$. In particular, $\mathcal{LH}^2(Q,R)=\mathfrak{TCL}(R,Q)/\approx$ with $\lhd_\lambda$ and $\rightharpoonup_\lambda$ trivial provides a theoretical answer to the $\mathbb{C}[\partial]$-split extension problem. This
cohomological type object is denoted by $\mathcal{LHC}^2(Q,R)$.

Let $R=\mathbb{C}[\partial]V$ and $Q=\mathbb{C}[\partial]x$. By Theorem \ref{t5},
$\mathcal{LHC}_R^2(Q,R)$ and $\mathcal{LHC}^2(Q,R)$ can be characterized by flag datums of $R$ with
$h_\lambda(\cdot,\partial)=g_\lambda(\cdot,\partial)=0$. In this case, the flag datums $\mathcal{F}_1(R)$ and
$\mathcal{F}_2(R)$ are the same. For convenience, we denote by $\mathcal{FC}(R)$ the set of flag datums
$(D_\lambda, T_\lambda, Q_0(\lambda,\partial), P(\lambda,\partial))$ of $R$ which satisfy (\ref{fg9}),
(\ref{fg10}) and
\begin{eqnarray}
\label{fc1}D_{\lambda+\mu}([a_\lambda b])=[a_\lambda D_\mu(b)]
-[b_\mu D_\lambda(a)],\\
\label{fc2}T_\mu([a_\lambda b])=[a_\lambda T_\mu(b)]-[D_\lambda(a)_{\lambda+\mu}b],\\
\label{fc3}T_\mu([a_\lambda b])=[a_\lambda T_\mu(b)]+[{T_\mu(a)}_{\lambda+\mu}b],\\
\label{fc4}P(\mu,\lambda+\partial)D_\lambda(a)=D_{\lambda+\mu}(D_\lambda(a))
+T_\mu(D_\lambda(a))-[a_\lambda Q_0(\mu,\partial)],\\
\label{fc5}D_{\lambda+\mu}(T_\mu(a))+D_{\lambda+\mu}(D_\lambda(a))=0,\\
\label{fc6}P(\lambda,-\lambda-\mu)P(\lambda+\mu,\partial)+P(\lambda,\mu+\partial)P(\mu,\partial)-P(\mu,\lambda+\partial)P(\lambda,\partial)=0.
\end{eqnarray}
Moreover, $(D_\lambda,
T_\lambda, Q_0(\lambda,\partial), P(\lambda,\partial))$ $\equiv$ ( $D_\lambda^{'},$  $T_\lambda^{'},$ $Q_0^{'}(\lambda,\partial),$ $P^{'}(\lambda,\partial))$
if and only if
there exist $u_0\in R$ and $\beta\in \mathbb{C}\backslash\{0\}$ such that (for all $a\in R$)
\begin{gather}
\label{cf1}D_\lambda(a)=[a_\lambda u_0]+\beta D_\lambda^{'}(a),\\
\label{cf2}T_\lambda(a)=[{u_0}_\lambda a]+\beta T_\lambda^{'}(a),\\
\label{cf3}P(\lambda,\partial)=\beta P^{'}(\lambda,\partial),\\
\label{cf4}Q_0(\lambda,\partial)=[{u_0}_\lambda u_0]+\beta D_\lambda^{'}(u_0)
+\beta T_\lambda^{'}(u_0)+\beta^2Q_0^{'}(\lambda,\partial)-P(\lambda,\partial)u_0.
\end{gather}
Whereas $(D_\lambda,
T_\lambda, Q_0(\lambda,\partial), P(\lambda,\partial))$ $\approx$ ( $D_\lambda^{'},$  $T_\lambda^{'},$ $Q_0^{'}(\lambda,\partial),$ $P^{'}(\lambda,\partial))$
if and only if (\ref{cf1})-(\ref{cf4}) hold with $\beta=1$. Note that in a flag datum $(D_\lambda,
T_\lambda, Q_0(\lambda,\partial), P(\lambda,\partial))$, $T_\lambda$ is a conformal derivation and
$D_\lambda$ is a conformal anti-derivation.

\begin{lem}\label{s5l1}
Suppose that $R$ is a Leibniz conformal algebra with the trivial left-center. Then for any flag datum
$(D_\lambda,
T_\lambda, Q_0(\lambda,\partial), P(\lambda,\partial))$, $D_{-\lambda-\partial}=-T_\lambda$.
\end{lem}
\begin{proof}
By (\ref{fc2}) and (\ref{fc3}), we get $[D_\lambda(a)_{\lambda+\mu}b]+[{T_\mu(a)}_{\lambda+\mu}b]=0$.
Therefore, $[(D_{-\mu-\partial}(a)+T_\mu(a))_{\lambda+\mu}b]=0$. Since $R$ has the trivial left-center,
one has $D_{-\mu-\partial}(a)=-T_\mu(a)$ for any $a\in R$. Thus, $D_{-\lambda-\partial}=-T_\lambda$.
\end{proof}

\begin{pro}\label{sec5p1}
Let $R=\mathbb{C}[\partial]V$ be a Leibniz conformal algebra with the trivial left-center
and $Q=\mathbb{C}[\partial]x$. Then flag datums $(D_\lambda,
T_\lambda, Q_0(\lambda,\partial), P(\lambda,\partial))\equiv (D_\lambda^{'},
T_\lambda^{'}, Q_0^{'}(\lambda,\partial), P^{'}(\lambda,\partial))$ if and only if there exist $u_0\in R$ and $\beta\in \mathbb{C}\backslash\{0\}$ such that (\ref{cf2}), (\ref{cf3}) and the following hold
\begin{eqnarray}
\label{cf5} [a_\lambda u_0]=-[{u_0}_{-\lambda-\partial}a], \ \ \forall \ a\in R.
\end{eqnarray}
%for any $a\in R$.

Moreover, flag datums $(D_\lambda,
T_\lambda, Q_0(\lambda,\partial), P(\lambda,\partial))\approx(D_\lambda^{'},
T_\lambda^{'}, Q_0^{'}(\lambda,\partial), P^{'}(\lambda,\partial))$ if and only if there exists $u_0\in R$  such that (\ref{cf2}), (\ref{cf3}) with $\beta =1$ and (\ref{cf5}) are satisfied.
\end{pro}
\begin{proof}
 According to Theorem \ref{t5}, we only need to show that (\ref{cf1})-(\ref{cf4}) are equivalent to (\ref{cf2}), (\ref{cf3}) and (\ref{cf5}).

By Lemma \ref{s5l1}, $D_{-\lambda-\partial}=-T_\lambda$ and $D_{-\lambda-\partial}^{'}=-T_\lambda^{'}$. By (\ref{cf2}),
$T_\lambda(a)=[{u_0}_\lambda a]+\beta T_\lambda^{'}(a)$. Therefore, $D_\lambda(a)=-T_{-\lambda-\partial}(a)=-[{u_0}_{-\lambda-\partial}a]-\beta T_{-\lambda-\partial}^{'}(a)=-[{u_0}_{-\lambda-\partial}a]+\beta D_\lambda^{'}(a)$. By (\ref{cf1}), we can get (\ref{cf5}).

By (\ref{fg9}), we have
\begin{eqnarray}
\label{cf6}P(\lambda,-\lambda-\mu)T_{\lambda+\mu}(a)=T_\lambda(T_\mu(a))-[Q_0(\lambda,\partial)_{\lambda+\mu}a]-T_\mu(T_\lambda(a)),\\
\label{cf7}P^{'}(\lambda,-\lambda-\mu)T_{\lambda+\mu}^{'}(a)=T_\lambda^{'}(T_\mu^{'}(a))-[Q_0^{'}(\lambda,\partial)_{\lambda+\mu}a]-T_\mu^{'}(T_\lambda^{'}(a)).
\end{eqnarray}
Plugging (\ref{cf2}) into (\ref{cf6}), we obtain
\begin{gather}
P(\lambda,-\lambda-\mu)[{u_0}_{\lambda+\mu}a]+\beta^2P^{'}(\lambda,-\lambda-\mu)T_{\lambda+\mu}^{'}(a)
=([{u_0}_\lambda[{u_0}_\mu a]]-[{u_0}_\mu[{u_0}_\lambda a]])\nonumber\\
+\beta(T_\lambda^{'}([{u_0}_\mu a])-[{u_0}_\mu T_\lambda^{'}(a)])
-\beta(T_\mu^{'}([{u_0}_\lambda a])-[{u_0}_\lambda T_\mu^{'}(a)])+\nonumber\\
\label{cf8}\beta^2(T_\lambda^{'}(T_\mu^{'}(a))-T_\mu^{'}(T_\lambda^{'}(a)))-[Q_0(\lambda,\partial)_{\lambda+\mu}a].
\end{gather}
By Jacobi identity, (\ref{cf6}) and the fact that $T_\lambda$ is a conformal derivation, (\ref{cf8}) becomes
\begin{eqnarray*}
[{P(\lambda,\partial)u_0}_{\lambda+\mu} a]&=&\beta^2[Q_0^{'}(\lambda,\partial)_{\lambda+\mu} a]+[[{u_0}_\lambda u_0]_{\lambda+\mu}a]
+\beta[T_\lambda^{'}(u_0)_{\lambda+\mu}a]\\&&-\beta[T_\mu^{'}(u_0)_{\lambda+\mu}a]-[Q_0(\lambda,\partial)_{\lambda+\mu}a].
\end{eqnarray*}
Therefore, $[(P(\lambda,\partial)u_0-\beta^2Q_0^{'}(\lambda,\partial)-[{u_0}_\lambda u_0]-\beta T_\lambda^{'}(u_0)+\beta D_\lambda^{'}(u_0)+Q_0(\lambda,\partial))_{\lambda+\mu} a]=0 $ for any $a\in R$. Since $R$ has the trivial left-center,
we get (\ref{cf4}). The proof is finished.
\end{proof}

The following two corollaries are immediate.

\begin{cor}\label{sec5l2}
Let $R=\mathbb{C}[\partial]V$ be a Lie conformal algebra with the trivial left-center
and $Q=\mathbb{C}[\partial]x$  a Leibniz conformal algebra.
We obtain\begin{itemize}
\item[(1)] If $Q$ is abelian, then $\mathcal{LHC}_R^2(Q,R)\cong
\mathcal{FC}(R)/\equiv$, where $\equiv$ is the equivalent relation on $\mathcal{FC}(R)$ given by: $(D_\lambda, T_\lambda, Q_0(\lambda,\partial), 0)\equiv
(D_\lambda^{'}, T_\lambda^{'}, Q_0^{'}(\lambda,\partial), 0)$ if and only if there exists $\beta\in \mathbb{C}^{\ast}$ such that $T_\lambda -\beta T_{\lambda}^{'}\in CInn(R)$.

\item[(2)] If $Q$ is the Virasoro Lie conformal algebra, then $\mathcal{LHC}_R^2(Q,R)
=\mathcal{LHC}^2(Q,R)\cong \mathcal{FC}(R)/\approx$, where $\approx$ is the equivalent relation on $\mathcal{FC}(R)$ given by: $(D_\lambda, T_\lambda, Q_0(\lambda,\partial), \partial+2\lambda)$ $\equiv
(D_\lambda^{'}, T_\lambda^{'}, Q_0^{'}(\lambda,\partial),$ $ \partial+2\lambda)$ if and only if $T_\lambda -T_{\lambda}^{'}\in CInn(R)$.\end{itemize}
\end{cor}

%\begin{proof}
%Since $R$ is a Lie conformal algebra, (\ref{cf5})  naturally holds. Then this corollary can be obtained from Proposition \ref{sec5p1}.
%\end{proof}

\begin{cor}
Let $R=\mathbb{C}[\partial]V$ be a Lie conformal algebra with the trivial left-center
and $Q=\mathbb{C}[\partial]x$ a Leibniz conformal algebra which is either abelian or the Virasoro Lie conformal algebra.
If $CDer(R)=CInn(R)$, then $\mathcal{LHC}_R^2(Q,R)=\mathcal{LHC}^2(Q,R)=(0,0,0,0)$.
\end{cor}
%\begin{proof}
%It can be directly obtained from Corollary \ref{sec5l2}.
%\end{proof}

Finally, we present an example of computing $\mathcal{LHC}_R^2(Q,R)$ and $\mathcal{LHC}^2(Q,R)$
in which $R$ is a Leibniz conformal algebra with a nontrivial left-center.

\begin{ex}\label{ex1}
Let $R=\mathbb{C}[\partial]L\oplus \mathbb{C}[\partial]W$ be a Leibniz conformal algebra satisfying
\begin{eqnarray}
[L_\lambda L]=(\partial+2\lambda)L,~~~[L_\lambda W]=(\partial+2\lambda)W, ~~[W_\lambda L]=[W_\lambda W]=0.
\end{eqnarray}
This algebra was given in \cite{ZH}. Obviously, $R$ has a nontrivial left-center $\mathbb{C}[\partial]W$. Set
$Q=\mathbb{C}[\partial]x$, which is either abelian or the Virasoro Lie conformal algebra.

To compute $\mathcal{LHC}_R^2(Q,R)$ and $\mathcal{LHC}^2(Q,R)$,  we need to determine conformal derivations and conformal anti-derivations of $R$. It is not hard to show that all conformal derivations and conformal anti-derivations of $R$ are inner. Therefore, by the equivalence of
flag datums, we can assume that $T_\lambda =0$, $D_\lambda(W)=0$ and
\begin{eqnarray}\label{y1}
D_\lambda(L)=[L_\lambda(f(\partial)L+g(\partial)W)]
=f(\lambda+\partial)(\partial+2\lambda)L+g(\lambda+\partial)(\partial+2\lambda)W,
\end{eqnarray}
for some $f(\lambda), g(\lambda)\in \mathbb{C}[\lambda]$. By (\ref{fc2}), we get $[D_\lambda(L)_{\lambda+\mu}L]=0$. Thus, $f(\lambda)=0$ by \eqref{y1}.
Consequently, $D_\lambda(L)=g(\lambda+\partial)(\partial+2\lambda)W$. Set $u_0=g(\partial)W$.
% to make $D_\lambda$ be $0$.
 By (\ref{cf1}) and (\ref{cf2}), we can assume that $D_\lambda=T_\lambda=0$. By (\ref{fc4}), we get $Q_0(\lambda,\partial)=0$. Therefore, we come to the conclusion that $\mathcal{LHC}_R^2(Q,R)=\mathcal{LHC}^2(Q,R)=(0,0,0,0)$ when $Q$ is abelian, and
 $\mathcal{LHC}_R^2(Q,R)=\mathcal{LHC}^2(Q,R)=(0,0,0,\partial+2\lambda)$ when $Q$ is the Virasoro Lie conformal algebra.
\end{ex}

\subsection{Bicrossed products}
Let $\Omega(R, Q)=(\leftharpoonup_\lambda, \rightharpoonup_\lambda, \lhd_\lambda, \rhd_\lambda, $$ f_\lambda, $ $\{\cdot_\lambda \cdot\})$ be an extending datum of a Leibniz conformal algebra $R$ by a $\mathbb{C}[\partial]$-module $Q$ with  $f_\lambda$ trivial. In this case, we denote this extending structure simply by $\Omega(R, Q)=(\leftharpoonup_\lambda, \rightharpoonup_\lambda, \lhd_\lambda, \rhd_\lambda, $ $\{\cdot_\lambda \cdot\})$. By Theorem \ref{t1}, $\Omega(R, Q)=(\leftharpoonup_\lambda, \rightharpoonup_\lambda, \lhd_\lambda, \rhd_\lambda, $ $\{\cdot_\lambda \cdot\})$ is a
Leibniz extending structure if and only if $(Q,\{\cdot_\lambda \cdot\})$ is a Leibniz conformal algebra, $(Q,\rightharpoonup_\lambda,\lhd_\lambda)$ is a module of $R$, $(R,\leftharpoonup_\lambda,\rhd_\lambda)$ is a module of $Q$ and
they satisfy (L2), (L4), (L5), (L10) and (L12). The associated unified product is denoted by
$R\natural_{\leftharpoonup,\rightharpoonup}^{\lhd,\rhd}Q$ and we call it the \emph{bicrossed product}
of $R$ and $Q$. The $\lambda$-product on $R\natural_{\leftharpoonup,\rightharpoonup}^{\lhd,\rhd}Q$ is given by
\begin{eqnarray}
[(a,x)_\lambda (b,y)]=([a_\lambda b]+a\leftharpoonup_\lambda y+x\rhd_\lambda b, a\rightharpoonup_\lambda y+x\lhd_\lambda b+\{x_\lambda y\}),
\end{eqnarray}
for any $a$, $b\in R$ and $x$, $y\in Q$. Note that $R$ and $Q$ are subalgebras of $R\natural_{\leftharpoonup,\rightharpoonup}^{\lhd,\rhd}Q$.

By Theorem \ref{t2}, it is easy to get
\begin{pro}
Let $R$ and $Q$ be two Leibniz conformal algebras. Set
$E=R\oplus Q$ where the direct sum is the sum of $\mathbb{C}[\partial]$-modules. Suppose
that $E$ is  a Leibniz conformal algebra such that $R$ and $Q$ are two subalgebras of $E$. Then $E$ is isomorphic to
a bicrossed product $R\natural_{\leftharpoonup,\rightharpoonup}^{\lhd,\rhd}Q$.
\end{pro}

Similarly, bicrossed products %as a special case of unified product
can be used to give an answer to the following problem which is also a special case of $\mathbb{C}[\partial]$-split extending structures problem.\\

{\noindent\bf Problem 2} \emph{Given two Leibniz conformal algebras $R$ and $Q$. Set $E = R\oplus Q $ where
the direct sum is the sum of $\mathbb{C}[\partial]$-modules. Describe and classify all Leibniz conformal
algebra structures on $E$ up to isomorphisms which stabilize $R$ such that $R$ and $Q$ are two subalgebras of $E$.}\\

By the discussions that in Section 3, this problem can be solved by the cohomological type object $\mathcal{LH}_R^2(Q,R)=\mathfrak{TCL}(R,Q)/\equiv$ where in these Leibniz conformal extending structures, $f_\lambda$ is trivial. For convenience, we denote this
cohomological type object by $\mathcal{LHB}_R^2(Q,R)$. In particular, when
$R=\mathbb{C}[\partial]V$ and $Q=\mathbb{C}[\partial]x$, $\mathcal{LHB}_R^2(Q,R)$ can be described
by flag datums of $R$ with $Q_0(\lambda,\partial)=0$.

%Finally, an example of computing $\mathcal{LHB}_R^2(Q,R)$ is given.

\begin{ex}
Let $R$ be the Leibniz conformal algebra given in Example \ref{ex1} and $Q=\mathbb{C}[\partial]x$ the abelian Leibniz conformal algebra. Note that $\mathbb{C}[\partial]L$ is the Virasoro Lie conformal algebra and $(Q,\rightharpoonup_\lambda, \lhd_\lambda)$ is a module of $R$.
Obviously, $(Q,\rightharpoonup_\lambda)$ is a module of the subalgebra $\mathbb{C}[\partial]L$ of $R$.
According to the representation theory of the Virasoro Lie conformal algebra, $h_\lambda(L,\partial)=0$ or
$h_\lambda(L,\partial)=\partial+\alpha\lambda+\beta$ for some $\alpha$, $\beta\in \mathbb{C}$.

Suppose that $h_\lambda(L,\partial)=0$. Setting $a=L$ and $b=W$ in (\ref{fg1}), we can directly get $h_\lambda(W,\partial)=0$. Then a bicrossed product of $R$ and $Q$ is also a crossed product. By Example \ref{ex1},  all flag datums of $R$ are equivalent to $(0,0,0,0,0)$.

%Therefore, we only need to consider the case when
Now consider the case $h_\lambda(L,\partial)=\partial+\alpha\lambda+\beta$. Letting
$a=W$ and $b=L$ in (\ref{fg1}), one can get $h_\lambda(W,\partial)=0$.
In the following we compute the flag datums of $R$.

 First, we compute $\mathcal{F}_1(R)$. By Example \ref{ex1}, all conformal derivations of $R$ are inner. Therefore, by the equivalence of flag datums, we can set $T_\lambda=0$. Then by (\ref{fg3}), one can assume that
$D_\lambda(L)=g_1(\lambda,\partial)W$ and $D_\lambda(W)=g_2(\lambda,\partial)W$ for some $g_1(\lambda,\partial)$ and $g_2(\lambda,\partial)\in \mathbb{C}[\lambda,\partial]$.
Note that $P(\lambda,\partial)=0$ and $Q_0(\lambda,\partial)=0$. Then it can be directly obtained from (\ref{fg6}) that $D_\lambda(W)=0$. Setting $a=b=L$ in (\ref{fg2}), one can get
\begin{gather}
\label{eb1}(\lambda-\mu)g_1(\lambda+\mu,\partial)=g_1(\mu,\lambda+\partial)(\partial+2\lambda)
+(\partial+\lambda+\alpha\mu+\beta)g_1(\lambda,\partial)\\
-g_1(\lambda,\mu+\partial)(\partial+2\mu)
-(\partial+\mu+\alpha\lambda+\beta)g_1(\mu,\partial).\nonumber
\end{gather}
Letting $\lambda=0$ in (\ref{eb1}) gives
\begin{eqnarray}
\beta g_1(\mu,\partial)=(\partial+\alpha\mu+\beta)g_1(0,\partial)-(\partial+2\mu)g_1(0,\mu+\partial)
\end{eqnarray}

{\rm (A1)} If $\beta\neq 0$, then
$g_1(\mu,\partial)=(\partial+\alpha\mu+\beta)g_1(\partial)-(\partial+2\mu)g_1(\mu+\partial)$ with $g_1(\partial)= \frac{g_1(0,\partial)}{\beta}\in \mathbb{C}[\partial]$. %Taking it into (\ref{eb1}), (\ref{eb1}) naturally holds.
By the discussions above, $\mathcal{F}_1(R)$ can be seen the set of  flag datums such as
$(h_\lambda^{\alpha,\beta}(\cdot,\partial), D_\lambda^{g_1},0, 0, 0)$ where $h_\lambda^{\alpha,\beta}(L,\partial)=\partial+\alpha\lambda+\beta$ and $h_\lambda^{\alpha,\beta}(W,\partial)=0$, $D_\lambda^{g_1}(L)=g_1(\lambda,\partial)W$ and $D_\lambda^{g_1}(W)=0$, where
$g_1(\lambda,\partial)$ is a solution of (\ref{eb1}).
 %If $\beta\neq 0$, $g_1(\mu,\partial)=(\partial+\alpha\mu+\beta)g_1(\partial)-(\partial+2\mu)g_1(\mu+\partial)$. Therefore,
 By Theorem \ref{t5},
$(h_\lambda^{\alpha,\beta}(\cdot,\partial), D_\lambda^{g_1},0, 0, 0)$ is equivalent to $(h_\lambda^{\alpha,\beta}(\cdot,\partial), 0,0, 0, 0)$ by setting $u_0=-g_1(\partial)W$ in \eqref{tt1}.

{\rm (A2)} If $\beta= 0$, $(h_\lambda^{\alpha,0}(\cdot,\partial),  D_\lambda^{g_1},0, 0, 0)$ is equivalent to $(h_\lambda^{\alpha^{'},0}(\cdot,\partial), D_\lambda^{g_1^{'}},0, 0, 0)$ if and only if
$\alpha=\alpha^{'}$ and there exist $\gamma\in \mathbb{C}$ and $g(\partial)\in \mathbb{C}[\partial]$ such that
\begin{eqnarray}
g_1(\lambda,\partial)=(\partial+2\lambda)g(\lambda+\partial)-(\partial+\alpha\lambda)g(\partial)+\gamma g_1^{'}(\lambda,\partial).
\end{eqnarray}

Therefore, by the discussion above, $\mathcal{F}_1(R)$ contains the three kinds of flag datums including $(0,0,0,0,0)$. Moreover, by Theorem \ref{t5}, the three kinds of flag datums are not equivalent
to each other.

Next, we compute $\mathcal{F}_2(R)$.

Set $D_\lambda(L)=f_1(\lambda,\partial)L+g_1(\lambda,\partial)W$
and $D_\lambda(W)=f_2(\lambda,\partial)L+g_2(\lambda,\partial)W$. By (\ref{fg12}), $T_\mu(a)+D_{-\mu-\partial}(a)\in \mathbb{C}[\partial,\mu]W$. Therefore, we can assume that $T_\lambda(L)=-f_1(-\lambda-\partial,\partial)L+g_3(\lambda,\partial)W$
and $T_\lambda(W)=-f_2(-\lambda-\partial,\partial)L+g_4(\lambda,\partial)W$.
 By (\ref{fg13}), we get
\begin{eqnarray}
\label{eb2}f_1(\lambda,\mu+\partial)((1-\alpha)\partial-\alpha\mu+\beta)=f_1(\lambda,-\lambda-\mu)(\partial+\alpha(\lambda+\mu)+\beta).
\end{eqnarray}
Therefore,
\begin{eqnarray}
f_{1}(\lambda, \partial)=\left\{
\begin{array}{ll}
0, \\
f_1(\lambda)(\partial+\lambda+\beta), & \alpha=1,\\
f_1(\lambda), & \alpha= 0,
\end{array}
\right.
\end{eqnarray}
where $f_1(\lambda)\neq 0$.

%by (\ref{eb2}), we can get that when $\alpha=1$, $f_1(\lambda,\partial)=f_1(\lambda)(\partial+\lambda+\beta)$ for some $f_1(\lambda)\neq 0$; when $\alpha= 0$, $f_1(\lambda,\partial)=f_1(\lambda)$ for some $f_1(\lambda)\neq 0$; $f_1(\lambda,\partial)=0$.

Suppose that $f_1(\lambda,\partial)\neq 0$. Therefore,
\begin{eqnarray}
D_\lambda(L)=\left\{
\begin{array}{ll}
f_1(\lambda)(\partial+\lambda+\beta)L+g_1(\lambda,\partial)W, & \alpha=1,\\
f_1(\lambda)L+g_1(\lambda,\partial)W, & \alpha= 0,
\end{array}
\right.
\end{eqnarray}
and
%when $\alpha=1$, $D_\lambda(L)=f_1(\lambda)(\partial+\lambda+\beta)L+g_1(\lambda,\partial)W$; when $\alpha=0$, $D_\lambda(L)=f_1(\lambda)L+g_1(\lambda,\partial)W$.
\begin{eqnarray}
T_\lambda(L)=\left\{
\begin{array}{ll}
-f_1(-\lambda-\partial)(-\lambda+\beta)L+g_3(\lambda,\partial)W, & \alpha=1,\\
-f_1(-\lambda-\partial)L+g_3(\lambda,\partial)W, & \alpha= 0.
\end{array}
\right.
\end{eqnarray}
%when $\alpha=1$, $T_\lambda(L)=-f_1(-\lambda-\partial)(-\lambda+\beta)L+g_3(\lambda,\partial)W$; when $\alpha=0$, $T_\lambda(L)=-f_1(-\lambda-\partial)L+g_3(\lambda,\partial)W$.
%Taking them into (\ref{fg17}), we can obtain that
However, when $\alpha=1$, (\ref{fg17}) does not hold for $a=L$. By (\ref{fg17}) again, $f_1(\lambda)=f_1$ for some
$f_1\in \mathbb{C}\setminus\{0\}$ when $\alpha=0$. But (\ref{fg1}) does not hold in this case.
%when $D_\lambda(L)=f_1L+g_1(\lambda,\partial)W$ for some
%$f_1\in \mathbb{C}\setminus\{0\}$.

Therefore, %by the discussion above,
$D_\lambda(L)=g_1(\lambda,\partial)W$, $T_\lambda(L)=g_3(\lambda,\partial)W$. With a similar discussion, one can get $D_\lambda(W)=g_2(\lambda,\partial)W$ and $T_\lambda(W)=g_4(\lambda,\partial)W$.

Setting $a=W$ and $b=L$ in (\ref{fg2}), one can immediately obtain that $D_\lambda(W)=0$.
By (\ref{fg16}),
$T_\lambda(T_\mu(W))=T_\mu(T_\lambda(W))$. Therefore, $g_4(\mu,\lambda+\partial)g_4(\lambda,\partial)
=g_4(\lambda,\mu+\partial)g_4(\mu,\partial)$. Consequently, we can assume that $g_4(\mu,\partial)=g_4(\mu)$ for some $g_4(\mu)\in \mathbb{C}[\partial]$.
Setting $a=L$ and $b=W$ in (\ref{fg11}), we can obtain
\begin{eqnarray}
(\partial+2\lambda+\mu)g_4(\mu,\partial)=g_4(\mu,\lambda+\partial)(\partial+2\lambda)-((\alpha-1)\lambda-\mu+\beta)g_4(\lambda+\mu,\partial).
\end{eqnarray}
It follows that $\mu g_4(\mu)=((1-\alpha)\lambda+\mu-\beta)g_4(\lambda+\mu)$. Therefore, $g_4(\mu)=0$ or
$g_4(\mu)=k$ for some $k\in \mathbb{C}\setminus\{0\}$ when $\alpha=1$ and $\beta=0$, i.e. $T_\lambda(W)=0$ or $T_\lambda(W)=k_1W$ for some $k_1\in \mathbb{C}\setminus\{0\}$ when $\alpha=1$ and $\beta=0$.

{\rm (B1)} Suppose that $T_\lambda(W)=k_1W$ for some $k_1\in \mathbb{C}\setminus\{0\}$ with $\alpha=1$ and $\beta=0$. By (\ref{fg6}), we can get $D_\lambda(L)=0$. According to (\ref{fg16}) with $a=L$, one can obtain
$g_3(\mu,\lambda+\partial)=g_3(\lambda,\mu+\partial)$. Therefore, we can assume that $g_3(\lambda,\partial)
=g_3(\lambda+\partial)$ for some $g_3(\lambda)\in \mathbb{C}[\lambda]$.
Taking it into (\ref{fg11}) with $a=b=L$, we have $g_3(\mu+\partial)=g_3(\lambda+\mu+\partial)$, which gives
$g_3(\lambda)=k_2$ for some $k_2\in \mathbb{C}$. %By the discussion above,

In this case, all flag datums of the second kind of $R$ are of the form $(h_\lambda^{1,0}(\cdot,\partial), 0, T_\lambda^{k_1, k_2}, 0, 0)$, where $h_\lambda^{1,0}(L,\partial)=\partial+\lambda$, $h_\lambda^{1,0}(W,\partial)=0$,
$T_\lambda^{k_1, k_2}(L)=k_2W$ and $T_\lambda^{k_1, k_2}(W)=k_1W$ with $k_1\in \mathbb{C}\setminus\{0\}$ and $k_2\in \mathbb{C}$. By Theorem \ref{t5},  $(h_\lambda^{1,0}(\cdot,\partial), 0, T_\lambda^{k_1, k_2}, 0, 0)\equiv  (h_\lambda^{1,0}(\cdot,\partial), 0, T_\lambda^{k_1^{'}, k_2^{'}}, 0, 0)$ if and only if
there exists some $c\in \mathbb{C}\setminus\{0\}$, such that $k_1=ck_1^{'}$, $k_2=ck_2^{'}$.

{\rm (B2)} Suppose that $T_\lambda(W)=0$. Setting $a=b=L$ in (\ref{fg11}), we have
\begin{gather}
\label{eb3}(\partial+\mu+2\lambda)g_3(\mu,\partial)=g_3(\mu,\lambda+\partial)(\partial+2\lambda)
-((1-\alpha)\partial-\alpha\mu+(1-\alpha)\lambda+\beta)g_1(\lambda,\partial)\\
-((\alpha-1)\lambda-\mu+\beta)g_3(\lambda+\mu,\partial).\nonumber
\end{gather}
Letting $\lambda=0$ in the above equality, one can get
\begin{eqnarray}
\beta g_3(\mu,\partial)=-((1-\alpha)\partial-\alpha\mu+\beta)g_1(0,\partial).\end{eqnarray}

If $\beta\neq 0$, %$g_3(\mu,\partial)=-((1-\alpha)\partial-\alpha\mu+\beta)\frac{g_1(0,\partial)}{\beta}$. %Set $g_1(\partial)=\frac{g_1(0,\partial)}{\beta}$. We get
$g_3(\mu,\partial)=-((1-\alpha)\partial-\alpha\mu+\beta)g_1(\partial)$, with $g_1(\partial)=\frac{g_1(0,\partial)}{\beta}$.
%Since $g_1(\lambda,\partial)$ satisfies (\ref{eb1}), by the above discussion, we can get the following result.
In this case, $\mathcal{F}_2(R)$ is the set of  flag datums of the form
$(h_\lambda^{\alpha,\beta}(\cdot,\partial), D_\lambda^{g_1}, T_\lambda^{g_3}, 0, 0)$ where $h_\lambda^{\alpha,\beta}(L,\partial)=\partial+\alpha\lambda+\beta$ and $h_\lambda^{\alpha,\beta}(W,\partial)=0$, $D_\lambda^{g_1}(L)=g_1(\lambda,\partial)W$,  $D_\lambda^{g_1}(W)=0$,  $T_\lambda^{g_3}(L)=g_3(\lambda,\partial)W$ and $T_\lambda^{g_3}(W)=0$, where
$g_1(\lambda,\partial)$ is a solution of (\ref{eb1}) and $g_3(\lambda,\partial)$ is a solution of (\ref{eb3}). When $\beta\neq 0$, by Theorem \ref{t5},
$(h_\lambda^{\alpha,\beta}(\cdot,\partial), D_\lambda^{g_1}, T_\lambda^{g_3}, 0, 0)$ is equivalent to $(h_\lambda^{\alpha,\beta}(\cdot,\partial), 0,0, 0, 0)$ by setting $u_0=-g_1(\partial)W$ in Theorem \ref{t5}.

When $\beta= 0$, $(h_\lambda^{\alpha,0}(\cdot,\partial), D_\lambda^{g_1}, T_\lambda^{g_3}, 0, 0)$ is equivalent to $(h_\lambda^{\alpha^{'},0}(\cdot,\partial), D_\lambda^{g_1^{'}},T_\lambda^{g_3^{'}}, 0, 0)$ if and only if
$\alpha=\alpha^{'}$ and there exist $\gamma\in \mathbb{C}$ and $g(\partial)\in \mathbb{C}[\partial]$ such that
\begin{eqnarray}
g_1(\lambda,\partial)=(\partial+2\lambda)g(\lambda+\partial)-(\partial+\alpha\lambda)g(\partial)+\gamma g_1^{'}(\lambda,\partial),\\
g_3(\lambda,\partial)=\gamma g_3^{'}(\lambda,\partial)+((1-\alpha)\partial-\alpha\lambda)g(\partial).
\end{eqnarray}

According to the discussion above, $\mathcal{F}_2(R)$ contains the three kinds of flag datums, and also they are not equivalent to each other.

Therefore, $\mathcal{LHB}_R^2(Q,R)$ can be described by the six kinds of flag
datums above under the corresponding equivalent relations.
\end{ex}

{\bf Acknowledgments}
{This work was supported by the Zhejiang Provincial Natural Science Foundation of China (No. LQ16A010011), the National Natural Science Foundation of China (No. 11501515, 11871421).}

\end{document}